\documentclass[a4paper, 11pt]{article}
\usepackage{amsbsy}
\usepackage{amsmath,amsfonts,amssymb,amsthm,mathrsfs}
\usepackage[a4paper,vmargin={3.5cm,3.5cm},hmargin={2.5cm,2.5cm}]{geometry}
\usepackage[font=sf, labelfont={sf,bf}, margin=1cm]{caption}
\usepackage{graphicx,graphics}
\usepackage{epsfig}
\usepackage{latexsym}
\usepackage[applemac]{inputenc}
\usepackage{ae,aecompl}
\usepackage{soul}
\usepackage[scaled]{helvet}
\usepackage{courier}
\usepackage[english]{babel}
 \usepackage[pdfpagemode=UseNone,bookmarksopen=false,colorlinks=true]{hyperref}
\usepackage{pstricks}
\usepackage{enumerate}
\usepackage{tikz}
\usetikzlibrary{matrix}
\usepackage{bbm}

\hypersetup{ colorlinks=true, linkcolor=blue, citecolor=blue,
  filecolor=blue, urlcolor=blue}

\DeclareMathOperator{\one}{\mathbbm{1}} 

\usepackage[comma, sort&compress]{natbib}

\theoremstyle{plain} 
    \newtheorem{theorem}{Theorem}
    \newtheorem{lemma}[theorem]{Lemma}
    \newtheorem{proposition}[theorem]{Proposition}
    
    \newtheorem{claim}[theorem]{Claim}

\theoremstyle{definition} 

    \newtheorem{remark}[theorem]{Remark}

\DeclareMathOperator{\R}{\mathbb{R}}
\DeclareMathOperator{\Q}{\mathbb{Q}}
\DeclareMathOperator{\Z}{\mathbb{Z}}
\DeclareMathOperator{\N}{\mathbb{N}}
\usepackage{mathtools}

\DeclareMathOperator{\De}{d}

\newcommand{\bR}{\mathbf{R}}

\newcommand{\I}{\mathcal I}

\newcommand{\f}{\frac}  

\newcommand{\eps}{\epsilon}
\newcommand{\eq}[1]{\begin{equation#1}}
\newcommand{\eeq}[1]{\end{equation#1}}
\newcommand{\eqa}[1]{\begin{eqnarray#1}}
\newcommand{\eeqa}[1]{\end{eqnarray#1}}
\newcommand{\vr}{\varphi}

\numberwithin{equation}{section}

\title{\vspace{-2cm}\textbf{Extremes of some Gaussian random interfaces} \vspace{0.5cm}}

\date{}

\author{Alberto Chiarini \thanks{Technische Universit\"at, Berlin, Germany\hfill  \texttt{chiarini@math.tu-berlin.de}}\qquad Alessandra Cipriani    \thanks{Weierstrass Institute, Berlin, Germany\hfill  \texttt{Alessandra.Cipriani@wias-berlin.de}}\qquad Rajat Subhra Hazra\thanks{ Indian Statistical Institute,
 Kolkata, India\hfill \texttt{rajatmaths@gmail.com}}}

\begin{document}
 \maketitle

 \begin{abstract}
In this article we give a general criterion for some dependent Gaussian models to belong to maximal domain of attraction of Gumbel, following an application of the Stein-Chen method studied in \cite{AGG}. We also show the convergence of the associated point process. As an application, we show the conditions are satisfied by some of the well-known supercritical Gaussian interface models, namely, membrane model, massive and massless discrete Gaussian free field, fractional Gaussian free field.
 \end{abstract}

\global\long\def\E{\mathsf{E}}

\global\long\def\Ex{\mathrm{E}}

\global\long\def\var#1{\mathsf{Var}\left[#1 \right]}

\global\long\def\cov#1{\mathsf{Cov}\left[#1 \right]}

\global\long\def\prob{\mathsf{P}}

\global\long\def\pro{\mathbb{P}}

\global\long\def\bprob{\mathbf{P}}

\global\long\def\vr{\varphi}

\global\long\def\Exp#1{\mathbb{E}\left[#1\right]}
\global\long\def\varp#1{\mathbb{V}\mathrm{ar}\left[#1 \right]}

\global\long\def\abs#1{\left\lvert #1\right\rvert }
 \global\long\def\sgn{\mathrm{{sgn\}}}}
 \global\long\def\e{\mathrm{e}}
$\global\long\def\O#1{\mathrm{O}\left(#1\right)}
$ $\global\long\def\o#1{\mathrm{o}\left(#1\right)}
$\global\long\def\la{}
 \global\long\def\f{\frac{}{}}
\global\long\def\R{\mathbb{R}}
\global\long\def\Z{\mathbb{Z}}
\global\long\def\Q{\mathbb{Q}}
\global\long\def\N{\mathbb{N}}
\global\long\def\De{\mathrm{d}}

\maketitle

\section{lntroduction}\label{sec: Introduction}
The main aim of this article is to study the fluctuations of the extremal process associated to a class of random interface models including the {\it discrete Gaussian free field} (DGFF) and the \textit{membrane model} (MM). These models have gained prominence in recent years due to links to other theories of statistical mechanics and probability such as the branching random walk, the branching Brownian motion, and the Schramm-Loewner evolution, just to name a few. It is known that in terms of infinite volume Gibbs measures these models undergo a phase transition: for instance, the DGFF has critical dimension two, that is, in $d\ge 3$ the associated infinite volume measure exists. After~\cite{BDG} determined the leading order of the maximum and~\cite{Daviaud} investigated the high points in $d=2$, the extremal process for the DGFF was studied at the critical dimension by \cite{BrDiZe, BisLou} and in the supercritical case it was handled recently in \cite{CCH2015, CCH2015b}. It was shown that the ideas used in critical DGFF can be extended to encapsulate many log-correlated models (see \citet[Theorem 1.3]{ding:roy:ofer}) so it is natural to ask similar questions for the supercritical case. The main aim of this article is to extend the ideas introduced in \cite{CCH2015} to derive a general criterion for Gaussian fields to belong to the maximal domain of attraction of the Gumbel distribution. We also check that these conditions hold for some known fields. The examples we consider here require bounds derived from the infinite volume Gibbs measure.

Criteria for one-dimensional Gaussian fields to belong to the maximal domain of attraction of the Gumbel distribution are classical (\citet[Section~3.3.3]{EmbKluMik}, \citet{Berman}). Extending them to a higher-dimensional setting is however non trivial. The main result of this article is Theorem~\ref{thm:zero}, which gives some conditions on the field that are relatively easy to check for the models we consider below. We use the Stein-Chen method for Poisson approximation and in particular interesting results from \cite{AGG}. While the DGFF case was dealt in \cite{CCH2015}, the massive DGFF in $d\ge 1$ is new and follows relatively easily. The model which is slightly more difficult to analyse (for various reasons which we will point out later) is the membrane model.  The membrane model undergoes a phase transition in $d=4$. Its extremal process has not been extensively studied yet. The first order behavior of the maxima in the critical dimension was determined in \cite{Kurt_d4} and the high points were studied in \cite{Cip13}. The fluctuations of the maxima are still an open question in any dimension. In this article we study the maxima and also the extremal point process in the supercritical dimensions (that is, in $d\ge 5$). We also study the fractional discrete Gaussian free field, namely the centered Gaussian field whose covariance is given by the Green's function of an isotropic random walk. In this case it is known that there are long-range interactions for the random walk in contrast to what happens for the DGFF, which has nearest-neighbour interactions. The entropic repulsion for this model was studied by~\cite{BDZ95}; see also~\cite{caputo_thesis} for further results on large deviations of the Gibbs measure.

\paragraph{Outline of the article}In Section~\ref{sec:extrema_Gauss} we state our main conditions and the results. In Theorem~\ref{thm:zero} we deal with maxima and give conditions for the scaling limit to be Gumbel. We show that the argument for maxima can be extended to prove that the extremal point processes converge to a Poisson random measure in Theorem~\ref{thm:zero:pp}. In Section~\ref{sec:examples} we give a brief introduction to the models which satisfy the conditions of Theorem~\ref{thm:zero}. The rest of the article is devoted to the proofs. In Section~\ref{sec:proof:maintheorem} we prove Theorem~\ref{thm:zero} and Theorem~\ref{thm:zero:pp}. In Section~\ref{sec:proof:examples} we show that the models satisfy the assumptions of the main theorems.

\subsection{Extremes of Gaussian fields}\label{sec:extrema_Gauss}
Denote by $V_N:=[0,\,n-1]^d\cap\Z^d$, $n\in \N$, the centered box of side-length $n$ and volume $N=n^d$, and given $\delta>0$, which we fix now for the rest of the paper, let us denote by $V_N^\delta$ to be the {\it bulk} of $V_N$, namely
$$
V_N^\delta:=\left\{\alpha\in V_N:\,\|\alpha-\gamma \|>\delta N^{1/d},\,\text{for all}\;\gamma\in \Z^d\setminus V_N\right\}.
$$
Let $\bprob_N$ and $\bprob$ be two centered Gaussian probability measures on $\R^{\Z^d}$. Let $(\vr_\alpha)_{\alpha\in \Z^d}$ be the canonical coordinate process for either $\bprob_N$ or $\bprob$. One can think for example of $\bprob_N$ as the finite volume measure. We indicate with $g_N(\cdot,\,\cdot)$ resp. $g(\cdot,\,\cdot)$ the covariance matrices for $\bprob_N$ resp. $\bprob$. We use the notation $g_N(\alpha):=g_N(\alpha,\,\alpha)$ and, for a stationary field with covariance matrix $g$, we indicate with a slight abuse of notation $g(\alpha):=g(0,\,\alpha)$.\\
In the following theorem $\mathbb P$ is either $\bprob_N$ or $\bprob$.
\begin{theorem}\label{thm:zero}
Let us define the centering and scaling by
\begin{equation}\label{eq:cs}
b_N:=\sqrt{g(0)}\left[\sqrt{2 \log N}-\frac{\log \log N+\log(4\pi)}{2\sqrt{2 \log N}}\right], \qquad a_N:=\frac{g(0)}{b_N}.
\end{equation}
We make the following assumptions:
\let\myenumi\theenumi
\let\mylabelenumi\labelenumi
\renewcommand{\theenumi}{A\myenumi}
\renewcommand{\labelenumi}{{\rm (\theenumi)}}
\begin{enumerate}
\item Let the measure $\bprob$ be stationary and $g(0)\in (0,\,+\infty)$. Assume further that $\lim_{\|\alpha\|\to+\infty}g(\alpha)= 0$.\label{item:decreasing}
\item We have that 
\[
g(\alpha,\,\beta)+\o{(\log N)^{-1}} \le g_N(\alpha,\,\beta)\le g(\alpha,\,\beta)+\o{(\log N)^{-1}}
\]
for $\alpha,\,\beta\in V_N^\delta$. Moreover $g_N(\alpha,\,\alpha)\le g(0,\,0)+\o{(\log N)^{-1}}$ for $\alpha\in V_N$. All the error terms are uniform in $\alpha,\,\beta$. \label{item:less}
\item Define  $\kappa:=1- \sup_{\alpha\in \Z^d} g(\alpha)/g(0)$. For all $\alpha\in V_N$, let $B_\alpha:=B\left(\alpha,\,s_N\right)$ be the ball around $\alpha$ of radius $s_N$, with $s_N$ satisfying $s_N=\o{ N^{\frac{\kappa}{d(2-\kappa)}}}$.
Let $K:=V_N\setminus B_\alpha$. Define $\mu_\alpha= \mathbb E[ \vr_\alpha|\mathcal F_K]$ where $\mathcal F_K=\sigma(\vr_\beta: \, \beta\in K)$. We have that
 $$\lim_{N\to\infty} \sup_{\alpha\in V_N^\delta}\varp{ \mu_\alpha}(\log N)^{2+\theta}=0 \;\text{ for some $\theta>0.$}$$\label{item:last}
\end{enumerate}
Then with the scaling of \eqref{eq:cs}
$$
\lim_{N\to +\infty}\pro\left(\frac{\max_{\alpha\in V_N}\varphi_\alpha-b_N}{a_N}<z\right)=\exp(-\mathrm{e}^{-z})
$$
uniformly for all $z\in \R$.
\end{theorem}
\begin{remark}
\begin{enumerate}
\item In Condition~\eqref{item:last} it not clear a priori that $\kappa$ is strictly greater than 0; as a part of the proof we will show that $\kappa\in (0,\,1]$ under Assumptions~\eqref{item:decreasing}-\eqref{item:decreasing}.
\item Condition~\eqref{item:less} is only needed in the finite volume case, that is, for the convergence of the maximum under $\bprob_N$.
\item In most of the examples we will deal with later, $s_N$ in Condition~\eqref{item:last} will be of order $ (\log N)^{T}$ for some $T>0$ and hence $s_N=\o{ N^{\frac{\kappa}{d(2-\kappa)}}}$ will be automatically satisfied.
\end{enumerate}
\end{remark}
\begin{remark}[Open question]
The extension to non-Gaussian fields is an interesting point one may raise. We are relying on the Gaussian structure when we apply Lemma~\ref{corol:conditional}, and also when we require stationarity in the form of covariance stationarity. We may guess that in the non-Gaussian setting one assumes that the field can be written as sum of two independent subfields with a suitable decay of covariances, that it is stationary in a strong sense and of bounded variances. Under these hypotheses the result still holds, with some amount of work. More general random walk representations (Helffer-Sj{\"o}strand) could help for example in the case of convex potentials.
\end{remark}

Our next result extends the convergence of the maximum to point process convergence. It is well-known that from the point process convergence one can derive the maximum, however we would still like to study the latter alone at first, as the estimates used in the proof are easier to present and can easily be implemented in the convergence of point processes. Let $E=[0,1]^d\times (-\infty,+\infty]$ and $V_N$ as in the beginning of this Subsection. Let $\varepsilon_x(\cdot)$, $x\in E$, be the point measure that gives mass one to a set containing $x$ and zero otherwise.

\begin{theorem}\label{thm:zero:pp}
Let $(\vr_\alpha)_{\alpha\in \Z^d}$ be a centered Gaussian field as in Theorem~\ref{thm:zero} satisfying Assumptions \eqref{item:decreasing}--\eqref{item:last}; then if $a_N$ and $b_N$ are as in \eqref{eq:cs},
\begin{equation}\label{eq:def:pp:general}
\eta_n(\cdot) :=\sum_{\alpha \in V_N} \varepsilon_{\left(\frac{\alpha}{n}, \frac{\vr_\alpha-b_N}{a_N}\right)}(\cdot)\overset{d}\rightarrow \eta
\end{equation}
where $\eta$ is a Poisson random measure on $E$ with intensity measure given by $\De t \otimes\left( \mathrm{e}^{-z} \De z\right)$, $ \De t \otimes \De z$ is the Lebesgue measure on $E$ and $\overset{d}\rightarrow $ is the convergence in distribution on $\mathcal M_p(E)$, the set of (Radon) point measures on $E$ endowed with the topology of vague convergence.
\end{theorem}
\begin{remark}
For the membrane model, the correct state space is in fact $E=[-1/2,\,1/2]^d\times (-\infty,\,+\infty]$ due to the choice of $V_N$. The computations can be carried through without any issue.
\end{remark}

The main idea of the proof is
\begin{itemize}
\item[$\spadesuit$] to reduce ourselves to show convergence in the bulk of $V_N$, that is, sufficiently inside the box so as to neglect possible effects of the boundary;
\item[$\spadesuit$] secondly, we will use a convergence-of-types theorem to show that ``forgetting'' the random variables close to the boundary does not affect the scaling limit;
\item[$\spadesuit$] finally, the core of the proof consists in applying a version of the Stein-Chen method that will allow us to control the total variation distance between the law of the maximum and the Gumbel distribution.
\end{itemize}

\section{Examples}\label{sec:examples}

In this section we give some examples which satisfy the conditions of Theorem~\ref{thm:zero}. We
will handle
\begin{itemize}
\item[$\spadesuit$] the membrane model in $d\ge 5$,
\item[$\spadesuit$] the massive Gaussian free field in $d\ge 1$,
\item[$\spadesuit$] the discrete fractional free field with index $s$, $s\in (0,\,\min\left\{2,\,d\right\})$,
\item[$\spadesuit$] the discrete Gaussian free field in $d\ge 3$.
\end{itemize}
For convenience we are going to recall some of the basic properties of these models which will be helpful for proving Assumption~\eqref{item:decreasing}--\eqref{item:last}. Note that the last case was treated in \cite{CCH2015} and since this work extends it, the supercritical DGFF meets these conditions. We do not discuss this model in further details here.
\subsection{The membrane model}\label{subsec:MM}
In this section we briefly recall the definition of the membrane model (see also \cite{Sakagawa,Kurt_thesis}). Let $V_N:=[-n/2,\,n/2]^d\cap \Z^d$. The membrane model is the random interface with zero boundary conditions outside $V_N$ whose distribution is given by
\begin{equation}
 \label{def:field}
\mathbf P_N(\De\varphi)=\frac{1}{Z_N}\exp\left(-\frac{1}{2}\sum_{\alpha\in \Z^d}(\Delta\varphi_\alpha)^2\right)\prod_{\alpha\in V_N}\De\varphi_\alpha
\prod_{\alpha\in V_N^{\mathrm c}}\delta_0(\De\varphi_\alpha),
\end{equation}
where $\Delta$ is the discrete Laplacian. Here $Z_N$ is the normalizing constant. Note that by re-summation, the law $\mathbf P_N$ of the field is the law of the centered Gaussian field on $V_N$ with covariance matrix
$$G_N(\alpha,\beta)=G_{V_N}(\alpha,\beta):=\cov{\varphi_\alpha,\varphi_\beta}=\left(\Delta_N^2\right)^{-1}(\alpha,\beta).$$
Here, $\Delta_N^2(\alpha,\beta)=\Delta^2(\alpha,\beta)\one_{\{\alpha,\,\beta\in V_N\}}$ is the Bilaplacian with 0-boundary conditions outside $V_N$. It can also be seen as a Gaussian field on $V_N$ whose covariance matrix $G_N$ satisfies, for $\alpha\in V_N,$
$$\left\{\begin{array}{lr}
\Delta^2 G_N(\alpha,\beta)=\delta(\alpha,\, \beta),& \beta \in V_N\label{eq:uno} \\
G_N(\alpha,\beta)=0, & \beta \in \partial_2 V_N.\nonumber
\end{array}\right.$$
Here we denote $\delta(\alpha,\,\cdot)$ the Dirac mass at $\alpha$ and $\partial_2V_N:=\{\beta\in V_N^{\mathrm c}:\,\exists \gamma\in V_N:\,  \|\beta-\gamma\|\leq 2\}$. It is known that in $d\ge 5$ there exists $\bprob$ on $\R^{\Z^d}$ such that $\bprob_N\to \bprob$ in the weak topology of probability measures (\citet[Proposition 1.2.3] {Kurt_thesis}).  Under $\bprob$, the canonical coordinates $(\vr_\alpha)_{\alpha\in \Z^d}$ form a centered Gaussian process with covariance given by
$$G(\alpha,\beta)= \Delta^{-2}(\alpha,\beta)= \sum_{\gamma\in \Z^d} \Delta^{-1}(\alpha,\gamma)\Delta^{-1}(\gamma,\beta)= \sum_{\gamma\in \Z^d} \Gamma(\alpha,\gamma) \Gamma(\gamma,\beta),$$
where $\Gamma$ denotes the covariance of the DGFF. $\Gamma$ has an easy representation in terms of the simple random walk $(S_n)_{n\ge 0}$ on $\Z^d$ given by
$$\Gamma(\alpha, \,\beta)=\sum_{m\ge 0} \mathrm P_\alpha[ S_m=\beta]$$
($\mathrm P_\alpha$ is the law of $S$ starting at $\alpha$). This entails that
\begin{equation}\label{eq:cov_membrane}
G(\alpha,\,\beta)= \sum_{m\ge 0} (m+1) \mathrm P_\alpha[ S_m=\beta]=\mathrm E_{\alpha,\beta }\left[ \sum_{\ell,\,m=0}^{+\infty} \one_{\left\{S_m= \tilde S_\ell\right\} }\right]
\end{equation}
where $S$ and $\tilde S$ are two independent simple random walks started at $\alpha$ and $\beta$ respectively.  First one can note from this representation that $G(\cdot,\,\cdot)$ is translation invariant. Hence we shall write $G(\alpha,\,\beta)=G(\alpha-\beta,\,0)=G(\alpha-\beta)$ with a slight abuse of notation, and the variance will be denoted by $G(0)$. The existence of the infinite volume measure in $d\ge 5$ gives that $G(0)<+\infty$. Using the above one can derive the following property of the covariance:
\begin{lemma}[{\citet[Lemma 5.1]{Sakagawa}}]\label{lemma: covariance:mm}
\eq{}\label{eq:cov:mm}
\lim_{\|\alpha\|\to+\infty}\frac{G(\alpha)}{\|\alpha\|^{4-d}}=\eta_2
\eeq{}
where
$$
\eta_2=(2\pi)^{-d}\int_0^{+\infty}\int_{\R^d}\exp\left(\iota\langle \zeta,\, \theta\rangle-\frac{\|\theta\|^4 t}{4\pi^2}\right)\De\theta\De t
$$
for any $\zeta\in \mathbb{S}^{d-1}$.
\end{lemma}
Just as the DGFF enjoys the spatial Markov property, the membrane model does too. In the finite volume case it was shown in \cite{Cip13}. The results extend easily by the DLR formalism to the infinite volume measure:
\begin{proposition}[Markov property]\label{prop: MP}
Let $(\psi_\alpha)_{\alpha\in\Z^d}$ and $(\vr_\alpha)_{\alpha\in \Z^d}$ be the finite and infinite volume membrane model under the measures $\bprob_N$ and $\bprob$ respectively.
\begin{itemize}
\item[(a)]  Finite volume (\citet[Lemma 2.2]{Cip13}): let $B \subseteq V_N$. Let $\mathcal F_B:=\sigma(\psi_\gamma,\,\gamma \in V_N\setminus B)$. Then
\eq{}\label{eq:cip13}
 \{\psi_\alpha\}_{\alpha\in B} \stackrel{d}{=} \left\{\mathbf E_N\left[\psi_\alpha|\mathcal F_B\right]+ \psi'_\alpha\right\}_{\alpha\in B}
\eeq{}
where ``$\stackrel{d}{=}$'' indicates equality in distribution, and, under $\mathbf P_N(\cdot)$, $\psi'_\alpha$ is independent of $\mathcal F_B$. Also $\{\psi'_\alpha\}_{\alpha \in B}$ is distributed as the membrane model with 0-boundary conditions outside B.
\item[(b)] Infinite volume: let $K \Subset \Z^d$ be a finite subset, $U:=K^{\mathrm c}$ and for $\alpha\in \Z^d$ define $\mu_\alpha:= \mathbf E[\vr_\alpha|\mathcal F_K]$ where $\mathcal F_K:=\sigma(\vr_\gamma: \gamma \in K)$. Also define $\psi_\alpha:=\vr_\alpha-\mu_\alpha$. Then $\mu_\alpha$ is $\mathcal F_K$-measurable, moreover under $\bprob$ we have that $\psi_\alpha$ is independent of $\mu_\alpha$ and distributed as $\bprob_U$, that is,
$$\mathbf E[ \psi_\alpha\psi_\beta]= G_U(\alpha,\beta).$$
\end{itemize}
\end{proposition}
Proposition~\ref{prop: MP} (b) will be shown later in the paper. The spatial Markov property is a crucial tool in the analysis of Gaussian models. For the DGFF it has been widely been applied for example to show properties in the percolation of level sets,  see~\cite{ASS, DrePF, PFASS}. We should stress that our proof of Proposition~\ref{prop: MP} (b) follows the ideas of \citet[Lemma 1.2]{PFASS}, although due to the lack of random walk representation for $G_U$ some more effort is required to achieve the result. In fact, we highlight the main challenges one encounters in handling the membrane model in contrast to the DGFF:
\begin{itemize}
\item[$\spadesuit$] the covariance of the finite volume measure lacks a random walk representation and hence many estimates rely on discrete potential theory as introduced in \cite{Kurt_d4,Kurt_d5}.
\item[$\spadesuit$] The finite volume Gaussian field can be negatively correlated and hence applications of the FKG inequality are not possible.
\end{itemize}

\subsection{Massive free field}\label{subsec:intro:massive}
The second model we consider is the massive free field. It is called so since it has an an Hamiltonian of the form
\[
-\sum_{\alpha \in \Z^d}(\nabla \varphi_\alpha)^2-m\sum_{x\in \Z^d}\varphi_\alpha^2
\]
where $m>0$ is a parameter called the ``mass'' and $\nabla$ is the discrete gradient.
For detailed properties of this model one may refer to \citet[Section~3.3]{Funaki}, \citet[Section~8.5]{veleniknotes} or to the recent article by~\cite{PFmassive}. Here we will review the key features we need, namely the random walk representation of its covariances.

Let $\vartheta\in (0,1)$ be fixed. We consider the graph $\Z^d\cup \{\ast\}$ where $\{\ast\}$ is a cemetery state. On this state space we define a Markov chain $S_n$ with transition probabilities as follows: $$p_{x, y}= \frac1{2d} (1-\vartheta) \one_{\left\{\alpha\sim \beta\right\}}, \quad p_{\alpha, \ast} =\vartheta, \quad p_{\ast, \ast}=1.$$
Let the canonical law of the chain starting at $\alpha\in \Z^d$ be denoted by $\prob^\alpha_\vartheta$ and its expectation by $\E^\alpha_\vartheta$.
Given $U\subset \Z^d$, $\prob^\alpha_{\vartheta, U}$ indicates the law of the chain starting at $\alpha\in \Z^d$ either killed uniformly at rate $\vartheta$ or when first entering $U$.  When $U=\emptyset$, we denote $\prob^{\alpha}_{\vartheta,\emptyset}= \prob^\alpha_{\vartheta}$.  The Green's function $g_{\vartheta, U}$ of this walk is
$$g_{\vartheta, U}(\alpha,\beta)=\sum_{n\ge 0} \prob^{\alpha}_{\vartheta, U}[ S_n=\beta]=\sum_{n\ge 0} (1-\vartheta)^n \prob^\alpha_0[ S_n=\beta, n<H_U]$$
where $H_U=\inf\{n\ge 0: \, S_n\in U\}$. When $U=\emptyset$, one writes $g_\vartheta(\alpha,\beta)= g_\vartheta(\alpha-\beta, 0)= g_\vartheta(\alpha-\beta)$. Take $U\subset \Z^d$ and $K\subseteq U^{\mathrm c}$, then by the strong Markov property one has
\begin{equation}\label{eq:massiveMP:green}
g_{\vartheta, U}(\alpha,\beta)= g_{\vartheta, U\cup K}(\alpha,\beta)+ \sum_{\gamma\in K} \prob^\alpha_{\vartheta, U}[ H_K<+\infty, X_{H_K}=\gamma] g_{\vartheta, U}(\gamma,\beta).
\end{equation}

Let us consider $\Z^d$ with $d\ge 1$ and $\vartheta\in (0,1)$. First we consider the massive free field on $\Z^d$, that is with $U=\emptyset$. For $\vartheta \in (0,1)$ we denote by $\bprob_\vartheta$ the law on $\R^{\Z^d}$ of the massive free field, under which  $\vr=(\vr_x)_{x\in \Z^d}$ are distributed as a centered Gaussian field with covariance
$$\mathbf E_\vartheta[ \vr_x \vr_y ]= g_{\vartheta}(x,y).$$
A consequence of~\eqref{eq:massiveMP:green} is that Proposition~\ref{prop: MP} holds for the massive Gaussian free field, see for example \citet[Lemma 1.1]{PFmassive}. Conditions for a random field to inherit the Markov property from the underlying Markov processes of the covariance are given also in \cite{Dyn80}.

\subsection{Fractional fields}\label{subsec:DFGF}
In this subsection we describe the fractional fields arising out of the Green's function which are local times of an isotropic stable law. The model has been dealt with in several works, for example in \cite{BDZ95} where the entropic repulsion event was studied, and in \cite{caputo_thesis}. Let $d\ge 1$ and $q_s$ be the density of the symmetric isotropic stable law on $\R^d$ for some $0<s<\min\left\{2,\,d\right\}$. This means that the characteristic function of $q_s$ is given by
$$\int_{\R^d} \e^{\iota \langle t,\, x\rangle} q_s(x) \De x=\e^{-\rho \|t\|^s} $$
for $t\in \R^d$ and some $\rho>0$. Let $Q(\alpha,\beta)$ for $\alpha,\beta\in \Z^d$ be the transition matrix of an isotropic $\alpha$-stable random walk, that is,
\begin{equation}\label{eq:transition:fgf}
Q(\alpha,\beta)=\int_{V} q_s(x+ (\alpha-\beta)^{+})\De x
\end{equation}
where $V=[-1/2, 1/2]^d$ and for $\alpha=(\alpha_1,\ldots, \alpha_d)\in \Z^d$, we denote by $(\alpha)^{+}=(|\alpha_1|,\ldots, |\alpha|^d)$.
The Green's function corresponding to the above random walk is given by
$$G_s(\alpha,\beta)=\sum_{m=0}^{+\infty} Q^m(\alpha,\beta)=(1-Q)^{-1}(\alpha,\beta).$$
We consider the Gaussian interface model whose law on $\Lambda\Subset \Z^d$ is given by
$$\mathbf P_{s,\Lambda}(\De \vr)=\frac1{Z_\Lambda}\exp(- H_\Lambda(\vr)) \prod_{\alpha\in \Lambda}\De \vr_\alpha \prod_{\alpha\notin \Lambda} \delta_0(\vr_\alpha)$$
where the Hamiltonian is given by
$$H_\Lambda(\vr)= \frac12\sum_{\alpha,\beta\in \Lambda} \vr_\alpha(I- Q)_{\Lambda}(\alpha,\beta)\vr_\beta$$
and $(I-Q)_\Lambda= (\delta(\alpha,\beta)- Q(\alpha,\beta))_{\alpha,\beta\in \Lambda}$.
Let $G_{s,\Lambda}$ denote the corresponding Green's function for the killed random walk, that is,\begin{equation}\label{eq:dfgf:greenrep}
G_{s,\Lambda}(\alpha,\beta) = \mathsf E_\alpha\left[ \sum_{m=0}^{\tau_\Lambda-1} \one_{\{S_m=\beta\}}\right].
\end{equation}
Here $\mathsf P$ denotes the law of an isotropic random walk and $\tau_\Lambda$ denotes its corresponding exit time from $\Lambda$.
It follows by calculations similar to Lemma 5 of \cite{ofernotes} that $(I-Q)_{\Lambda}$ is symmetric and positive definite and hence the Gibbs measure exists by \citet[Proposition 13.13]{Georgii} and moreover it is a centered Gaussian field on $\R^\Lambda$ with covariance $\mathbf E_\Lambda [\vr_\alpha \vr_\beta]= G_{s,\,\Lambda}(\alpha,\beta)$. Note that using the characteristic function of $q_s$  and the fact that $s<d$ it follows that $\mathbf P_{s,\Lambda}\to \mathbf P$ as $\Lambda\uparrow \Z^d$. In this case the finite volume Gibbs measure also satisfies the DLR equation. A consequence of the Markov property of the field is that one can write, for $K\Subset \Z^d$ and $U=K^{\mathrm c}$,
\begin{equation}\label{eq:green:MP}
G_s(\alpha,\beta)= G_{s,U}(\alpha,\beta)+ \sum_{\gamma\in K} \prob_\alpha[ \tau_U<+\infty, S_{\tau_U}=\gamma]G_s(\gamma,\beta).
\end{equation}
Note that using the above decomposition for $K$ and letting $\mathcal F_K=\sigma(\vr_\beta: \,\beta\in K)$, one has from Lemma A.2 of \cite{caputo_thesis},
 \begin{equation}\label{eq:drift}
 \mu_\alpha=\mathbf E[ \vr_\alpha| \mathcal F_K]=\sum_{\beta\in K}\mathsf P_\alpha\left(H_{K}<+\infty,\,S_{H_{K}}=\beta\right)\vr_\beta,\quad \alpha\in \Z^d.
 \end{equation}
Recall $H_K:=\inf\left\{n\ge 0:\,S_n\in K \right\}$ and $S_n$ is the law of the isotropic random walk defined above.


It is known from \cite{BDZ95} that there exists $\omega_{s,d}>0$ such that
\begin{equation}\label{eq:dfgf:cov}
\frac{G_s(\alpha,\beta)}{\omega_{s,d} \|\alpha-\beta\|^{s-d}}\to 1\; \text{ as $\|\alpha-\beta\|\to +\infty$}.
\end{equation}
Note that since $s<d$ we have the $G_s(\alpha,\beta)\to 0$ as $\|\alpha-\beta\|\to + \infty$.

\section{Proofs of the result of Section~\ref{sec:extrema_Gauss}}\label{sec:proof:maintheorem}
\subsection{Reminder on the Stein-Chen method}
For the reader's convenience in this subsection we recall the results from \cite{AGG} which we use crucially in our proofs, and also in order to fix the notations for the subsequent Sections.

Let $A$ be a countable index set and $(X_\alpha)_{\alpha\in  A}$ be a sequence of (possibly dependent) Bernoulli random variables of parameter $p_\alpha$. Let $W:=\sum_{\alpha\in A}X_\alpha$ and $\lambda:=E[{W}]$. Now for each $\alpha$ let $B_\alpha\subseteq  A$ be a subset which we consider a ``neighborhood'' of dependence for the variable $X_\alpha$, such that $X_\alpha$ is nearly independent from $X_\beta$ if $\beta\in A\setminus B_\alpha$. Set
$$
b_1:=\sum_{\alpha\in A}\sum_{\beta\in B_\alpha}p_\alpha p_\beta,
$$
$$
b_2:=\sum_{\alpha\in A}\sum_{\alpha\neq \beta\in B_\alpha}E\left[{X_\alpha X_\beta}\right],
$$
$$
b_3:=\sum_{\alpha\in A}E\left[{\left|E\left[{X_\alpha-p_\alpha\left|\right.\mathcal H_1}\right]\right|}\right]
$$
where
$$
\mathcal H_1:=\sigma\left(X_\beta:\,\beta\in A\setminus B_\alpha\right).
$$
\begin{theorem}[Theorem 1, \cite{AGG}]\label{thm:AGG}
Let $Z$ be a Poisson random variable with $E[{Z}]=\lambda$ and let $\|\cdot\|_{TV}$ be the total variation distance between probability measures. Then
$$
\|\mathcal L(W)-\mathcal L(Z)\|_{TV}\le 2(b_1+b_2+b_3)
$$
and
$$
\left|P(W=0)-\e^{-\lambda} \right|<\min\left\{1,\,\lambda^{-1}\right\}(b_1+b_2+b_3).
$$
\end{theorem}
The ``point process'' version of the Poisson approximation Theorem reads as follows:
\begin{theorem}[{\citet[Theorem 2]{AGG}}]\label{thm:AGG2}
Let $A$ be an index set. Partition the index set $A$ into disjoint non-empty sets $A_1,\,\ldots, \,A_k$. Let $(X_\alpha)_{\alpha\in A}$ be a dependent Bernoulli process with parameter $p_\alpha$. Let $(Y_\alpha)_{\alpha\in A}$ be independent Poisson random variables with intensity $p_\alpha$. Also let
$$W_j:= \sum_{\alpha\in A_j} X_\alpha \quad\text{ and }\quad Z_j:=\sum_{\alpha\in A_j} Y_\alpha\quad \text{ and}\quad\lambda_j:=  E[W_j]= E[Z_j].$$
Then
\begin{equation}\label{eq:errorbd}
\| \mathcal L( W_1, \ldots, W_k)- \mathcal L( Z_1,\ldots, Z_k)\|_{TV}\le 2\min\left\{1,\, 1.4 \left(\min \lambda_j\right)^{-1/2}\right\}(2b_1+2b_2+b_3)
\end{equation}
where $\mathcal L( W_1, \ldots, W_k) $ denotes the joint law of these random variables.
\end{theorem}
\subsection{Proof of Theorem~\ref{thm:zero}}
\subsubsection{Reduction to the bulk}\label{subsec:red}
We begin by noticing that convergence is required only ``well inside'' $V_N$. More precisely, set $u_N(z):=b_N+ a_N z$. Recall that in this Subsection $\vr_\alpha$ denotes the coordinate process under $\pro$ and we will use the notation $\mathbb{G}(\alpha)$ for the variance.
\begin{lemma}
Let us denote by $L_N= \max_{\alpha\in V_N^\delta}\vr_\alpha$ and $J_N=\max_{\alpha\in V_N\setminus V_N^\delta}\vr_\alpha$. Assume that
\begin{equation}\label{assumption:finite}
\lim_{N\to+\infty}\pro\left( L_N\le a_N z+ b_N\right)= \exp\left(-\e^{-z+d\log (1-2\delta)}\right).
\end{equation}
Then it follows that
$$\lim_{N\to+\infty}\pro\left( \max_{\alpha\in V_N} \vr_\alpha \le a_N z+ b_N\right)= \exp(-\e^{-z}).$$
\end{lemma}
\begin{proof}
First note that $\pro( \max_{\alpha\in V_N} \vr_\alpha\le u_N(z))= \pro( \max\{ L_N,\, J_N\} \le u_N(z))$. This implies that
$\pro(\max_{\alpha\in V_N} \vr_\alpha\le u_N(z)) \le \pro(L_N \le u_N(z)) \to \exp\left(-\e^{-z+d\log (1-2\delta)}\right)$ by \eqref{assumption:finite}. Now as $\delta \to 0$ this gives an upper bound. For a lower bound, using $P(A\cap B) \ge P(A)- P(B^{\mathrm c})$ we have
$$\pro\left(\max_{\alpha\in V_N} \vr_\alpha\le u_N(z)\right) \ge \pro\left( L_N\le u_N(z)\right) - \pro\left( J_N> u_N(z)\right).$$
Using the Mills ratio inequality
\eq{}\label{eq:Mills}
\left( 1-\frac1{t^2}\right)\frac{\e^{-t^2/2}}{\sqrt{2\pi}t}\le P\left(\mathcal N(0,\,1)>t \right)\le \frac{\e^{-{t^2/2}}}{\sqrt{2\pi}t},\quad t>0,
\eeq{}
we have
\begin{align*}
&\pro( J_N>u_N(z)) \le \sum_{\alpha\in V_N\setminus V_N^\delta} \pro( \vr_\alpha> u_N(z))\le\sum_{\alpha\in V_N\setminus V_N^\delta} \frac{\exp\left(- \frac{u_N(z)^2}{2 \mathbb{G}(\alpha)}\right)}{\sqrt{2\pi }u_N(z)} \sqrt{ \mathbb{G}(\alpha)}\\
&\le C\left|V_N\setminus V_N^\delta\right|\frac{\exp\left(- \frac{u_N(z)^2}{2g(0)}\right)}{\sqrt{2\pi }u_N(z)} \sqrt{ g(0)} \le  \left(1-(1-2\delta)^d\right) \to 0
\end{align*}
as $\delta\to 0$, where in the third inequality we have used Assumption~(\ref{item:less}) in the case $\pro=\bprob_N$. Combining this with \eqref{assumption:finite} and then letting $\delta\to 0$, one obtains the lower bound.
\end{proof}
Now we have to take care of the convergence on the bulk. The cardinality of the set we are working with will be slightly less than that of the original set and we need to be careful with it. More concretely, if we have $m_N:=|V_N^\delta|\approx (1-2\delta)^d N$ random variables, the next Lemma will allow us to derive the convergence of the maximum in $V_N$ on the scale $u_{N}(z)$ from that on $u_{m_N}(z)$. The following lemma is a consequence of \citet[Proposition 0.2]{Resnick} and for a detailed proof we refer to \citet[Lemma 8]{CCH2015}.
\begin{lemma}\label{lemma:cot}
Let $N\ge 1$, $F_N$ be a distribution function, and $m_N:=|V_N^\delta|$. Let $a_N$ and $b_N$ be as in~\eqref{eq:cs}. If $\lim_{N\to+\infty}F_N(a_{m_N}z+ b_{m_N})=\exp(-\e^{-z})$, then $$\lim_{N\to+\infty}F_N(a_Nz+b_N)=\exp\left(-\e^{-z+ d\log(1-2\delta)}\right).$$
\end{lemma}
%
\subsubsection{Proof of Theorem~\ref{thm:zero}}\label{subsec:pointwise}
We start with the proof of convergence in the bulk. Define, for all $\alpha\in V_N^\delta$, $p_\alpha:=P(\vr_\alpha>u_{m_N(z)})$, and
$$
X_\alpha=\one_{\left\{\varphi_\alpha>u_{m_N}(z)\right\}}\sim Ber(p_\alpha).
$$
We furthermore introduce
$
W:=\sum_{\alpha=1}^N X_\alpha.
$
Of course $W$ is closely related to the maximum since $\left\{\max_{\alpha\in V_N^\delta}\vr_\alpha\le u_{m_N}(z)\right\}=\left\{W=0 \right\}$. We will now fix $z\in \R$ and $\lambda:=\e^{-z}$.
Our main idea is to apply Theorem~\ref{thm:AGG}. To this scope we define $B_\alpha$ as in the assumptions of Theorem~\ref{thm:zero} and show that $b_1,\,b_2$ and $b_3$ as per Theorem~\ref{thm:AGG} converge to zero. \\
\indent i. $b_1=\sum_{\alpha\in V_N^\delta}\sum_{\beta\in B_\alpha}p_\alpha p_\beta$. Exploiting the fact that for a fixed $z$ one can choose $N$ large enough so that $u_{N}(z)>0$, it follows that
$$p_\alpha \stackrel{\eqref{eq:Mills}}{\le} \frac{\e^{-\frac{u_{m_N}(z)^2}{2g(0)}}}{\sqrt{2\pi}u_{m_N}(z)} \sqrt{g(0)}$$
where in the case $\pro=\bprob_N$ we can use Assumption~\eqref{item:less} (modulo paying the price of a multiplicative constant). By the equality $|V_N^{\delta}|= m_N$ we get that $$b_1\le m_N s_N^d
\left( \frac{\e^{-\frac{u_{m_N}(z)^2}{2g(0)}}}{\sqrt{2\pi}u_{m_N}(z)} \sqrt{g(0)}\right)^2\le  c m_N^{-1} s_N^d\e^{-2z}\stackrel{(\ref{item:last})}{=}\o{1}.$$
Note that in the last line we have used the weaker assumption $s_N=\o{m_N^{1/d}}$ compared to the Assumption~\eqref{item:last} on $s_N$.\\
\indent {ii.} $b_2=\sum_{\alpha\in A}\sum_{\alpha\neq \beta\in B_\alpha}\Exp{X_\alpha X_\beta}$. First we need to estimate the joint probability
$$\pro\left(\vr_\alpha>u_{m_N}(z), \vr_\beta>u_{m_N}(z)\right).$$
To do so, we will prove it suffices to treat the case where $\pro=\bprob$, and then see that the same term with $\bprob_N$ is ``almost'' the same when $N$ is large.
Denote the covariance matrix of the vector $(\vr_\alpha,\,\vr_\beta)$ by
$$\Sigma_2=\begin{bmatrix}
\mathbb G(\alpha) & \mathbb G(\alpha,\,\beta)\\
\mathbb G(\alpha,\,\beta) & \mathbb G(\beta) \\
\end{bmatrix}$$
for $\mathbb G$ the covariance matrix of $\pro$ (so it can be either $g$ or $g_N$).
Note that, for $w\in \R^2$, one has
$$w^T\Sigma_2^{-1} w=\frac{\mathbb{G}(\alpha)w_2^2+\mathbb{G}(\beta)w_1^2 -2\mathbb{G}(\alpha,\,\beta) w_1w_2}{\mathbb{G}(\alpha)\mathbb{G}(\beta)-\mathbb{G}(\alpha,\,\beta)^2} .$$
Using $\mathbf{1}:=(1,\,1)^T$ we denote by
$$\Delta_i:=u_{m_N}(z) \left(\mathbf{1}^T \Sigma_2^{-1}\right)_i,\quad i=1,\,2.$$
For $\pro=\bprob$, one can see immediately that for $i=1,\,2$
\eq{}\label{eq:expon}
\Delta_i=\frac{u_{m_N}(z) (g(0)-g(\alpha,\,\beta))}{g(0)^2-g(\alpha,\,\beta)^2}.
\eeq{}
On the other hand, for $\bprob_N$, using the lower bound given in Assumption~\eqref{item:less},
\begin{align}
\Delta_1&= u_{m_N}(z)\frac{ \mathbb{G}(\beta)-\mathbb{G}(\alpha,\,\beta)}{\mathbb{G}(\alpha)\mathbb{G}(\beta)-\mathbb{G}(\alpha,\,\beta)^2}\ge u_{m_N}(z)\frac{ (g(0)-g(\alpha,\,\beta))+\o{(\log N)^{-1}}}{g(0)^2-g(\alpha,\,\beta)^2+\o{(\log N)^{-1}}}\nonumber\\
&=u_{m_N}(z) \frac{g(0)-g(\alpha,\,\beta)}{g(0)^2-g(\alpha,\,\beta)^2}+\o{1},\label{eq:expo}
\end{align}
with the same estimate holding for $\Delta_2$.
Observe that the same argumentation shows also that
\eq{}\label{eq:er}u_{m_N}(z)^2\,\mathbf{1}^T \Sigma_2^{-1}\mathbf 1=u_{m_N}(z)^2\frac{2(g(0)-g(\alpha,\,\beta))}{g(0)^2-g(\alpha,\,\beta)^2}+\o{1}  \eeq{}
in the $\bprob_N$ case.
Exploiting an easy upper bound on bivariate Gaussian tails (see \cite{Savage}), we consider first the $\bprob$ case and obtain
\begin{align}
&\bprob(\vr_\alpha>u_{m_N}(z), \vr_\beta>u_{m_N}(z))\le \frac1{2\pi} \frac1{|\det \Sigma_2|^{1/2}\Delta_1\Delta_2}\exp\left(-\frac{u_{m_N}(z)^2}{2}  \mathbf{1}^T \Sigma_2^{-1}\mathbf 1\right)\nonumber\\
&\stackrel{\eqref{eq:expon}}{=} \frac1{2\pi} \frac{(g(0)+g(\alpha-\beta))^2}{(g(0)^2-g(\alpha-\beta)^2)^{1/2}u_{m_N}(z)^2}\exp\left(-\frac{u_{m_N}(z)^2}{2}  \frac{2(g(0)-g(\alpha-\beta))}{g(0)^2-g(\alpha-\beta)^2}\right) \nonumber \\
&\le \frac1{4\pi \log {m_N}} \frac{\left(1+ \frac{g(\alpha-\beta)}{g(0)}\right)^{3/2}}{\left(1-\frac{g(\alpha-\beta)}{g(0)}\right)^{1/2}} {m_N}^{-\frac{2g(0)}{g(0)+g(\alpha-\beta)}}\left(4\pi \log {m_N} \right)^{\frac{g(0)}{g(0)+g(\alpha-\beta)}}\e^{-\frac{2g(0)z}{g(0)+g(\alpha-\beta)}+\o{1}}\nonumber\\
&\le \frac{\left(1+ \frac{g(\alpha-\beta)}{g(0)}\right)^{3/2}}{\left(1-\frac{g(\alpha-\beta)}{g(0)}\right)^{1/2}} {m_N}^{-\frac{2g(0)}{g(0)+g(\alpha-\beta)}}\e^{-\frac{2g(0)z}{g(0)+g(\alpha-\beta)}+\o{1}}\label{eq:thisone}
\end{align}
where in the second-to-last inequality we used $u_{m_N}(z)^2= b_{m_N}^2+ 2g(0)z+g(0)^2 z^2/b_{m_N}^2$ and the bound of $b_{m_N}^2$
$$g(0)(2\log m_N-\log \log m_N-\log 4\pi)\le b_{m_N}^2\le 2g(0)\log m_N.$$
Note that in the second line, when dealing with $\pro=\bprob_N$, we can use \eqref{eq:expo} so that a multiplicative factor $(1+\o{1})$ would appear in the fraction and an additive $\o{1}$ would appear in the exponential due to the error terms in \eqref{eq:expo}, \eqref{eq:er}, but we can drop them without influencing the asymptotic. Hence \eqref{eq:thisone} is valid under $\pro$. We make now the following claim on the variance of the measure $\bprob$ (which we will prove in a moment):
\begin{claim}\label{claim:ferragosto}
$$ \sup_{\alpha\in \Z^d\setminus\{0\}}\frac{g(\alpha)}{g(0)}< 1.$$
\end{claim}
Now assuming it, we have that $1-\kappa:= \sup_{\alpha\in \Z^d\setminus\{0\}}g(\alpha)/g(0)\in [0,1)$.
From this we derive
\eq{}\label{eq:august}\frac{g(0)}{g(0)+g(\alpha-\beta)}\ge \frac1{2-\kappa} \qquad\text{ and } \qquad\frac{g(\alpha-\beta)}{g(0)+g(\alpha-\beta)}\le 1-\kappa.\eeq{}
We obtain thus
\begin{align*}\pro&\left( \vr_\alpha> u_{m_N}(z), \vr_\beta>u_{m_N}(z)\right)\\
&\stackrel{\eqref{eq:thisone},\,\eqref{eq:august}}{\le} \frac{(2-\kappa)^{3/2}}{\kappa^{1/2}} {m_N}^{-\frac{2}{(2-\kappa)}} \max\left( \e^{-2z}\one_{\left\{z\le 0\right\}},\, \e^{-2z/(2-\kappa)}\one_{\left\{z>0\right\}}\right).
\end{align*}
We get finally for some constants $c,\,c'>0$ depending only on $d$, $\delta$ and $\kappa$
\begin{align*}
b_2&\le c m_N s_N^d \frac{(2-\kappa)^{3/2}}{\kappa^{1/2}} {m_N}^{-\frac{2}{(2-\kappa)}}\max\left( \e^{-2z}\one_{\left\{z\le 0\right\}},\, \e^{-2z/(2-\kappa)}\one_{\left\{z>0\right\}}\right)\nonumber\\
&\le c' {m_N}^{-\frac{\kappa}{(2-\kappa)}} s_N^d \max\left( \e^{-2z}\one_{\left\{z\le 0\right\}}, \,\e^{-2z/(2-\kappa)}\one_{\left\{z>0\right\}}\right).
\end{align*}
Since $\kappa/(2-\kappa)>0$ and Assumption~(\ref{item:last}) holds, we have that $b_2=\o{1}$. What is left to conclude is to show Claim~\ref{claim:ferragosto}.
\begin{proof}[Proof of Claim~\ref{claim:ferragosto}]
The fact that $g(\alpha)\le g(0)$ follows from the Cauchy-Schwarz inequality; equality holds if and only if $\vr_{\alpha}=A\vr_{0}+B$ $\bprob$-a.~s. for some $A,\,B.$
By $\mathbf{E}\left[\vr_\alpha\right]=0$ we have $B=0$ and by $\mathbf E\left[\vr_{\alpha}^{2}\right]=\mathbf E\left[\vr_{0}^{2}\right]$ we
obtain $A=1$, in other words $\vr_\alpha=\vr_0$ a.-s. Since the field is stationary, we know
\[
\mathbf E\left[\vr_{0}\vr_{2\alpha}\right]\stackrel{assumption}{=}\mathbf E\left[\vr_{\alpha}\vr_{2\alpha}\right]=\mathbf E\left[\vr_{0}\vr_{\alpha}\right]\stackrel{assumption}{=}g(0).
\]
This implies that for $k\in\mathbb{N},$ $g(k\alpha)$ is a strictly positive constant, which contradicts the fact that $\lim_{\|\alpha\|\to+\infty}g(\alpha)=0$. For $\|\alpha\|$ large, by the asymptotics of $g(\alpha)$ we have $\sup_{\|\alpha\|>R}g(\alpha)/g(0)<1/2$
for some $R=R(1/2)$. Then let us consider the ball $B(0,R)$. There
we cannot have that there exists an $\alpha_{0}$ with $g(\alpha_{0})/g(0)=1$,
otherwise, as we have just seen, $g(k \alpha_0)$ would be constant for all $k\in \N$, so $\max_{\alpha\in B(0,R)}g(\alpha)/g(0)<1$.
We can take then
\[
\sup_{\alpha\in \Z^d\setminus\{0\}}\frac{g(\alpha)}{g(0)}\le\max\left\{ \frac12,\,\max_{\alpha\in B(0,R)}\frac{g(\alpha)}{g(0)}\right\}<1 .
\]
\end{proof}
\indent {iii.} $b_3=\sum_{\alpha\in V_N^\delta}\Exp{\left|\Exp{X_\alpha-p_\alpha\left|\right.\mathcal H_1}\right|}$.
It will be convenient to introduce another $\sigma$-algebra which strictly contains $\mathcal H_1=\sigma\left(X_\beta:\,\beta \in V_N^\delta\setminus B_\alpha \right)$, namely
$$
\mathcal H_2:=\sigma\left(\vr_\beta:\,\beta\in V_N\setminus B_\alpha\right).
$$
Using the tower property of the conditional expectation and Jensen's inequality
\begin{align*}
\Exp{\left|\Exp{X_\alpha-p_\alpha\left|\right.\mathcal H_1}\right|}&=\Exp{\left|\Exp{\Exp{X_\alpha-p_\alpha\left|\right.\mathcal H_2}\left|\right.\mathcal H_1} \right|}\\
&\le \Exp{\Exp{\left|\Exp{X_\alpha-p_\alpha\left|\right.\mathcal H_2}\right|\left|\right.\mathcal H_1}}=\Exp{\left|\Exp{X_\alpha-p_\alpha\left|\right.\mathcal H_2}\right|}.
\end{align*}
We now use the following Lemma which follows from~\citet[Lemma 4]{ofernotes} and \citet[Theorem 9.9]{Jan97}.
\begin{lemma}[Conditional laws]\label{corol:conditional}
Let $K:= V_N\setminus B_\alpha$ where $\alpha\in V_N^\delta$.  Define  $\psi_\alpha:= \vr_\alpha-\mu_\alpha,\,\alpha\in \Z^d$ as in Proposition~\ref{prop: MP}, then it follows that $\psi_\alpha$ and $\mu_\alpha$ are independent. Also the following properties hold:
\begin{enumerate}[(i)]
\item The law of $(\mu_\alpha)_{\alpha\in \Z^d}$ is that of a centered Gaussian with covariance matrix given by
\eq{*}\left(\sum_{\xi\in K}T(\alpha,\,\xi)\mathbb{G}(\beta,\,\xi)\right)_{\alpha,\,\beta\in \Z^d}
\eeq{*}
for some deterministic $T$.
\item The law of $\psi_\alpha$ under $\mathbb{P}(\cdot\,|\,\mathcal F_K)$ is that of a centered Gaussian such that $\psi_\alpha=0$ a.~s. if $\alpha\in K$, and with variance $\varp{\vr_\alpha}-\varp{\mu_\alpha}$ (in particular it is deterministic and does not depend on the values $(\vr_\alpha)_{\alpha\in K}$).
\end{enumerate}
\end{lemma}
\begin{proof}
{(i)} It follows from \citet[Lemma 4]{ofernotes} that
$$\mu_\alpha:= \Exp{\vr_\alpha|\mathcal F_K}=\sum_{\gamma\in K} T(\alpha,\gamma) \vr_\gamma,$$
where $T(\alpha,\,\gamma)=\sum_{\xi\in K} \mathbb{G}(\alpha,\,\xi) \mathbb{G}_K^{-1}(\xi,\,\gamma)$.   Note that for $\alpha\in K$ we have that $T(\alpha,\gamma)=\delta(\alpha,\gamma)$. A relevant observation is that $\psi_\alpha$ and $\mu_\alpha$ are independent: indeed for a $L^1$-random variable $\vr_\alpha$ and a sub-sigma-field $\mathcal F_K$, one has that $\vr_\alpha-\mathbb E[ \vr_\alpha|\mathcal F_K]$ is independent of any $\mathcal F_K$-measurable $Y$ \cite[Lemma A.55]{veleniknotes}. Now $\mu_\alpha=\sum_{\gamma\in K} T(\alpha,\gamma)\vr_\gamma$ is $\mathcal F_K$-measurable and hence is independent of $\psi_\alpha=\vr_\alpha-\mathbb E[ \vr_\alpha|\mathcal F_K]$.\\
{(ii)} The proof can be found in \citet[Theorem 9.9]{Jan97}.
\end{proof}
It is worth noticing that in some cases (for example the DGFF or in the MM) the law of $\psi_\alpha$ has a more explicit representation, and it is indeed that of a DGFF or MM with Dirichlet boundary conditions outside $B_\alpha$.\\
We will now prove a claim that will be of key importance in estimating the convergence of $b_3$.
\begin{lemma}\label{lemma:claim6} It holds that
$$\lim_{N\to+\infty}\sup_{\alpha\in V_N^{\delta}} \left(\frac{g(0)}{\varp{\psi_\alpha}}-1\right)u_{m_N}(z)^2=0.$$
\end{lemma}
\begin{proof}
With the help of Lemma~\ref{corol:conditional} (ii) and Assumption~\eqref{item:less} we have
$$
0\le \frac{g(0)}{\varp{\psi_\alpha}}-1
\le\frac{\varp{\mu_\alpha}+\o{(\log N)^{-1}}}{g(0)-\varp{\mu_\alpha}+\o{(\log N)^{-1}}}.
$$
Using Assumption~\eqref{item:last} it follows that $(\log N) \varp{\mu_\alpha}=\o{1}$ uniformly over all $\alpha \in V_N^\delta$ and hence the claim follows since $u_{m_N}(z)^2=C\log N+\o{\log N}$ for some constant $C>0$.
\end{proof}
We proceed by showing $b_3=\o{1}.$ Let us call $\pro_{K}$ the law of $\vr$ under the conditioning of $\mathcal H_2$. Noting that $\mu$ is $\mathcal H_2$-measurable, we see that $(\psi_\alpha)_{\alpha\in\Z^d\setminus K}$ under $\pro_{K}$ has covariance structure which was explicitly stated in Lemma~\ref{corol:conditional} (ii). We write\eqa{*}
\mathbb E[|\mathbb E[X_\alpha-p_\alpha |\mathcal H_2]|]&=&\mathbb E\left[\left|\mathbb E_K\left[\one_{\left\{\psi_\alpha+\mu_\alpha\ge u_{m_N}(z)\right\}}-p_\alpha\right]\right|\right]\\
&=&\mathbb E\left[\left|\mathbb E_{K}\left[\one_{\left\{\psi_\alpha+\mu_\alpha\ge u_{m_N}(z)\right\}}-p_\alpha\right]\right|\one_{\left\{\mu_\alpha>\left(u_{m_N}(z)\right)^{-(1+\theta)}\right\}}\right]\\
&+&\mathbb E\left[\left|\mathbb E_{K}\left[\one_{\left\{\psi_\alpha+\mu_\alpha\ge u_{m_N}(z)\right\}}-p_\alpha\right]\right|\one_{\left\{\mu_\alpha\le \left(u_{m_N}(z)\right)^{-(1+\theta)}\right\}}\right]=:T_1+T_2,
\eeqa{*}
with $\theta$ as in Assumption~\eqref{item:last}.
Let us handle $T_1$ first. Using the estimate $P\left(\left|\mathcal N(0,1)\right|>a\right)\le \e^{-{a^2/2}}$ for $a>0$,
we get that
\eq{}\label{eq:rate_zero}
\pro\left(|\mu_\alpha|>\left(u_{m_N}(z)\right)^{-1-\theta} \right)\le \e^{-(u_{m_N}(z))^{-2(1+\theta)}/(2\varp{\mu_\alpha})}.\eeq{}
Hence note that, using Assumption~\eqref{item:last}, we have, for some constants $C,\,C_1>0$,
\begin{align}
\sum_{\alpha\in V_N^\delta}& \pro\left(|\mu_\alpha|>\left(u_{m_N}(z)\right)^{-1-\theta} \right)\nonumber\\
&\le C\exp\left(-\log m_N\left[C_1 \left((\log m_N)^{2+\theta} \sup_{\alpha\in V_N^\delta} \varp{\mu_\alpha}\right)^{-1} -1\right]\right)\label{eq:after}
\end{align}
where we have used that $u_N(z)\sim b_N$ and $b_N\sim C\sqrt{2\log N}$. Hence we have $\sum_{\alpha\in V_N^\delta}T_1$ tends to zero.
Note that $T_2$ can be tackled using calculations similar to $b_3$ in \cite{CCH2015} and Lemma~\ref{lemma:claim6}. For completeness we provide a few details:
\begin{align}
\mathbb E&\left[\left| \mathbb P_K(\psi_\alpha+\mu_\alpha>u_{m_N}(z))-p_\alpha\right|\one_{\left\{|\mu_\alpha|\le\left(u_{m_N}(z)\right)^{-1-\theta} \right\}}\right]\nonumber\\
&=\Exp{\left(\mathbb P_K(\psi_\alpha+\mu_\alpha>u_{m_N}(z))-p_\alpha\right)\one_{\left\{|\mu_\alpha|\le\left(u_{m_N}(z)\right)^{-1-\theta} \right\}}\one_{\left\{p_\alpha<\mathbb P_K(\psi_\alpha+\mu_\alpha>u_{m_N}(z))\right\}}}\nonumber\\
&+\Exp{\left(p_\alpha-\mathbb P_K(\psi_\alpha+\mu_\alpha>u_{m_N}(z))\right)\one_{\left\{|\mu_\alpha|\le\left(u_{m_N}(z)\right)^{-1-\theta} \right\}}\one_{\left\{p_\alpha\ge\mathbb P_K(\psi_\alpha+\mu_\alpha>u_{m_N}(z))\right\}}}\nonumber\\
&=:T_{2,1}+T_{2,2}.\label{eq:treat}
\end{align}
We treat first $T_{2,\,2}$. Under the event $\left\{|\mu_\alpha|\le\left(u_{m_N}(z)\right)^{-1-\theta} \right\}$ one obtains
{\small\eqa{}
&&p_\alpha-\mathbb P_K(\psi_\alpha+\mu_\alpha>u_{m_N}(z))\nonumber\\
&&\stackrel{\eqref{eq:Mills},\,\text{(\ref{item:less})}}{\le} \frac{\sqrt{g(0)}\e^{-\frac{u_{m_N}(z)^2}{2g(0)}}}{\sqrt{2\pi}u_{m_N}(z)}-\left(1-\left(\frac{\sqrt{\varp{\psi_\alpha}}}{u_{m_N}(z)-\mu_\alpha}\right)^{2}\right)\frac{\sqrt{\varp{\psi_\alpha}}\e^{-\frac{(u_{m_N}(z)-\mu_\alpha)^2}{2 \varp{\psi_\alpha}}}}{\sqrt{2\pi }(u_{m_N}(z)-\mu_\alpha)}\nonumber\\
&&\le  \frac{\sqrt{g(0)}\e^{-\frac{u_{m_N}(z)^2}{2g(0)}}}{\sqrt{2\pi }u_{m_N}(z)}\left(1-\left(1+\o{1}\right)\frac{\sqrt{\varp{\psi_\alpha}}u_{m_N}(z)\e^{\left(1-\frac{g(0)}{\varp{\psi_\alpha}}\right)\frac{u_{m_N}(z)^2}{2g(0)}+\frac{\mu_\alpha u_{m_N}(z)}{\varp{\psi_\alpha}}-\frac{\mu_\alpha^2}{2\varp{\psi_\alpha}}}}{\sqrt{ g(0)}(u_{m_N}(z)-\mu_\alpha)} \right)\nonumber\\
&&=\frac{\sqrt{g(0)}\e^{-\frac{u_{m_N}(z)^2}{2g(0)}}}{\sqrt{2\pi }u_{m_N}(z)}\left(1-\left(1+\o{1}\right)\frac{\sqrt{\varp{\psi_\alpha}}u_N(z)\e^{\left(1-\frac{g(0)}{\varp{\psi_\alpha}}\right)\frac{u_{m_N}(z)^2}{2g(0)}+\o{1}}}{\sqrt{ g(0)}u_{m_N}(z)(1-u_{m_N}(z)^{-2-\theta})} \right).\label{eq:here2}
\eeqa{}}
In the last line we used the fact that $|\mu_\alpha|\le\left(u_{m_N}(z)\right)^{-1-\theta}$. By bounding the indicator functions by $1$,
$$
\Exp{\left(p_\alpha-\mathbb P_K(\psi_\alpha+\mu_\alpha>u_{m_N}(z))\right)\one_{\left\{|\mu_\alpha|\le\left(u_{m_N}(z)\right)^{-1-\theta} \right\}}\one_{\left\{p_\alpha\ge\mathbb P_K(\psi_\alpha+\mu_\alpha\le u_{m_N}(z))\right\}}}\le\eqref{eq:here2}.
$$
Now
\eq{}\label{eq:tap}
b_3\le \sum_{\alpha\in V_N^\delta}(T_1+T_2)\stackrel{\eqref{eq:after}}\le \sum_{\alpha\in V_N^\delta}T_1 +\o{1}=\sum_{\alpha\in V_N^\delta}T_{2,1}+\sum_{\alpha\in V_N^\delta}T_{2,2}+\o{1}.
\eeq{}
By \eqref{eq:here2} and Lemma~\ref{lemma:claim6},
\eqa{}\label{eq:tapp}\sum_{\alpha\in V_N^\delta}T_{2,2}\le m_N\frac{\sqrt{g(0)}\e^{-\frac{u_{m_N}(z)^2}{2g(0)}}}{\sqrt{2\pi }u_{m_N}(z)}\o{1}=\e^{-z+\o{1}}\o{1}.\label{eq:T_12} \eeqa{}
Analogously, $ \sum_{\alpha\in V_N^\delta}T_{2,\,1}=\o{1}$.
Plugging \eqref{eq:tapp} in \eqref{eq:tap}, one obtains $b_3=\o{1}$.
The argument to obtain uniform convergence is an immediate application of P\'olya's continuity Theorem
\cite[Satz I]{Polya}, since the limiting distribution function $\exp(-\e^{-z})$ is continuous.

\subsection{Proof of Theorem~\ref{thm:zero:pp}}
\begin{proof}
The proof follows ideas similar to \cite{CCH2015b} and hence we shall only briefly sketch the steps.

{\bf Step 1: reduction to bulk.}  For $\delta>0$, recall that the bulk is denoted by $V_N^\delta$. Let us denote by
$$\eta_n^\delta(\cdot)= \sum_{\alpha\in V_N^\delta} \varepsilon_{\left(\frac{\alpha}{n}, \frac{\vr_\alpha-b_N}{a_N}\right)}(\cdot).$$
It is well-known that $\mathcal M_p(E)$ is a complete, separable metric space endowed with a metric $d_p$. First we show that for any $\eps>0$,
\begin{equation}\label{eq:pp:joint}
\lim_{\delta\to 0}\limsup_{n\to+\infty}\pro\left[d_p\left(\eta_n, \eta_n^\delta\right)>\eps\right]=0.
\end{equation}
To prove~\eqref{eq:pp:joint} it is enough to show that, for any continuous, positive, compactly supported function $f$ on $E$ we have
$$\lim_{\delta\to 0}\limsup_{n\to+\infty}\pro\left[\left|\eta_n(f)- \eta_n^\delta(f)\right|>\eps\right]=0.$$
Since $f$ is compactly supported we can choose $z_0\in \R$ such that the support of $f$ is contained in $[0,\,1]\times (z_0,\,+\infty)$. Therefore
\begin{align*}
\Exp{\left|\eta_n(f)- \eta_n^{\delta}(f)\right|}&= \Exp{\left|\sum_{\alpha\in V_N\setminus V_N^{\delta} } f\left(\frac{\alpha}{n}, \frac{\vr_\alpha-b_N}{a_N}\right)\one_{\left\{\frac{\vr_\alpha-b_N}{a_N}>z_0\right\}}\right|}\\
&\le \sup_{z\in E}|f(z)| \sum_{\alpha\in V_N\setminus V_N^\delta} \pro\left[ \vr_\alpha>a_Nz_0+b_N\right]\\
&\overset{\eqref{eq:Mills}}\le C(1-(1-2\delta)^d)\e^{-z_0}
\end{align*}
Hence~\eqref{eq:pp:joint} follows by taking $\delta\to 0$.

{\bf Step 2: convergence together.} Let us denote by $\eta^\delta$ a Poisson random measure with intensity $\De t_{|_{[\delta,1-\delta]^d}}\otimes \left(\e^{-x}\De x\right)$ on $E$. Then it follows from the proof of Theorem 2 of \cite{CCH2015b} that $\eta^\delta\overset{d}\to \eta$. By \citet[Theorem 3.5]{ResnickHeavy} to complete the proof of Theorem~\ref{thm:zero:pp} it is enough to show that for any fixed $\delta>0$, $\eta_n^\delta\overset{d}\rightarrow \eta^\delta$ as $n\to+\infty$.

{\bf Step 3: Kallenberg's conditions.} To show the convergence of $\eta_n^\delta$ to $\eta^\delta$, it is enough to show the following two conditions due to Kallenberg \cite[Theorem 4.7]{KallenbergRan}.
 \begin{enumerate}
\item[i)] For any $A$ bounded rectangle in $[0,1]^d$ and $R=(x,y]\subset (-\infty,+\infty]$
$$\E[ \eta_n^\delta( A\times (x,y])]\to \E[ \eta^\delta( A\times (x,y])]=\left|A\cap [\delta,1-\delta]^d\right|( \e^{-x}-\e^{-y}).$$
We adopt the convention $\e^{-\infty}=0$ and the notation $|A|$ for the Lebesgue measure of $A$.
\item[ii)] For all $k\ge 1$, $A_1,\, A_2,\,\ldots,\, A_k$ disjoint rectangles in $[0,1]^d$ and $R_1,\,R_2,\,\ldots, \,R_k$, each of which is a finite union of disjoint intervals of the type $(x,\,y] \subset (-\infty,+\infty]$,
\begin{align}
&\prob\left( \eta_n^\delta( A_1\times R_1)=0,\, \ldots,\, \eta_n^\delta(A_k\times R_k)=0\right)\nonumber\\
&\to \prob\left(\eta^\delta( A_1\times R_1)=0, \,\ldots,\, \eta^\delta(A_k\times R_k)=0\right)=\exp\left(-\sum_{j=1}^k |A_j\cap[\delta,1-\delta]^d| \omega\left(R_j\right)\right)\label{eq:cond_two}
\end{align}
where $\omega(\De z):=\e^{-z}\De z$.
\end{enumerate}
Note that the first condition follows from~\eqref{eq:Mills} and the expansion of $(a_Nz+b_N)^2$, $z\in \R$. For the second condition we use Theorem~\ref{thm:AGG2}. To apply the result, let us denote by $\I_j:=n A_j\cap V_N^\delta$ and $\I=\I_1\cup\cdots \cup \I_k$. We are setting $B_\alpha:=B\left(\alpha,\,s_N\right)\cap \I$ where $s_N$ is as in Condition~\eqref{item:last}.  Let $X_\alpha=\one_{\left\{\frac{\vr_\alpha-b_N}{a_N}\in R_j\right\}}$ if $\alpha\in \I_j$ and $p_\alpha=\pro\left[ \frac{\vr_\alpha-b_N}{a_N}\in R_j\right]$. Let $W_j=\sum_{\alpha\in \I_j}X_\alpha$ and $Z_j$ be as in Theorem~\ref{thm:AGG2}. Then the left hand side of~\eqref{eq:cond_two} is given by $\pro\left[ W_1=0,\,\ldots, \,W_k=0\right]$ and the limit on the right side is $\prob\left[Z_1=0,\ldots, \,Z_k=0\right]$. Hence to show~\eqref{eq:cond_two} we only need to show that $b_1, b_2$ and $b_3$ go to zero. Since the proof of this fact is similar to the proof of Theorem~\ref{thm:zero} we leave out the details to avoid repetitions.
\end{proof}

\section{Proof of Assumptions for examples in Section~\ref{sec:examples}}\label{sec:proof:examples}

\subsection{Proofs of basic properties of membrane model}\label{sec: proof membrane}
We introduce some well-known notations of stopping times for the simple random walk $S_m$ which will be used throughout this section. Let $K$ be a subset of $\Z^d$. Let us recall $H_K:=\inf\{m\ge 0: \,S_m\in K\}$  and $\tau_K:=\inf\{ m\ge 0: \,S_m\notin K\}$ to be the first entrance time and first exit time for the walk.
In this section we show that the MM satisfies the properties of Theorem~\ref{thm:zero}. The following lemma shows monotonicity of the variance which becomes crucial in the proof of Assumption~(\ref{item:last}).
\begin{lemma}[Monotonicity of variances]
 \label{lemma:monotonicity_variances}
Let $A\subseteq B\Subset \Z^d.$ Then for all $\alpha\in A$
$$G_A(\alpha,\alpha)\leq G_B(\alpha,\alpha).$$
\end{lemma}
\begin{proof} We recall \eqref{eq:cip13}, that is, for a MM $\psi$ under $\bprob_B$
\[
 \{\psi_\alpha\}_{\alpha\in A} \stackrel{d}{=} \left\{\mathbf E_B\left[\psi_\alpha|\mathcal F_{B\setminus A}\right]+ \psi'_\alpha\right\}_{\alpha\in A}
\]
where $\psi'$ has the law of a MM on $A$ with zero boundary conditions in $B\setminus A$. Therefore
\[
G_B(\alpha,\,\alpha)-G_A(\alpha,\,\alpha)=\mathbf E_B\left[(\mathbf E_B\left[\psi_\alpha|\mathcal F_{B\setminus A}\right])^2\right]\ge 0.
\]
\end{proof}
The next lemma shows the Markov decomposition of the covariance function in terms of random walks.
\begin{proposition}\label{prop:G_decomposition}
 Let $A\subset \Z^d$, $A^{\mathrm c}:=\Z^d\setminus A$ and let $G$ be as in \eqref{eq:cov_membrane}. The equality
 \[
  G(\alpha,\,\beta)= \overline G_A(\alpha,\,\beta)+\Ex_\alpha\left(G(S_{H_{A^{\mathrm c}}},\,\beta)\one_{\{H_{A^{\mathrm c}}<+\infty\}}\right)+\Ex_\alpha\left( \one_{\{H_{A^{\mathrm c}}<+\infty\}} H_{A^{\mathrm c}} \Gamma(S_{H_{{A^{\mathrm c}}}},\beta)\right)
 \]
holds for all $\alpha,\,\beta\in \Z^d$. Here we denote $\overline G_A(\alpha,\beta)= \sum_{k=0}^{+\infty} (k+1)\mathrm P_\alpha\left(S_k=\beta, k< \tau_A \right)$ and $\Gamma(\alpha,\beta)=\sum_{k=0}^{+\infty} \mathrm P_\alpha( S_k=\beta)$.
\end{proposition}
\begin{proof}
 Let $\Theta_n$ be the canonical time shift on the space of nearest neighbor trajectories. Then we see that one can write, by Fubini's theorem and \eqref{eq:cov_membrane},
 \begin{align*}
 G(\alpha,\beta)&=\Ex_{\alpha}\left( \sum_{k=0}^{+\infty} (k+1) \one_{\{S_k=\beta\}}\right)\\
&=\Ex_\alpha\left(\sum_{k=0}^{\tau_A-1}(k+1)\one_{\{S_k=\beta\}}\right)+\Ex_\alpha\left(\sum_{k=\tau_A}^\infty(k+1)\one_{\{S_k=\beta\}} \one_{\{\tau_A<+\infty\}}\right)\\
&=\sum_{k=0}^{+\infty} (k+1) \mathrm P_\alpha( S_k=\beta, k<\tau_A)+\Ex_\alpha\left(\one_{\{\tau_A<+\infty\}}\sum_{k=0}^\infty(k+\tau_A+1)\one_{\{S_{k+\tau_A}=\beta\}}\right)\\
&=\overline G_A(\alpha,\beta)+ \Ex_\alpha\left( \one_{\{\tau_A<+\infty\}} \left( \sum_{k=0}^{+\infty} (k+1) \one_{\{S_k=\beta\}}\right)\circ \Theta_{\tau_A}\right)+\\
&+\Ex_\alpha\left(\tau_A \one_{\{\tau_A<+\infty\}}\left(\sum_{k=0}^{+\infty}\one_{\{S_k=\beta\}}\right)\circ \Theta_{\tau_A}\right)\\
&=\overline G_A(\alpha,\beta)+ \Ex_\alpha\left( \one_{\{\tau_A<+\infty\}} \Ex_{S_{\tau_A}}\left( \sum_{k=0}^{+\infty} (k+1) \one_{\{S_{k}=\beta\}}\right)\right)+\\
&+\Ex_\alpha\left(\tau_A \one_{\{\tau_A<+\infty\}}\Ex_{S_{\tau_A}}\left(\sum_{k=0}^{+\infty}\one_{\{S_k=\beta\}}\right)\right)\\
&=\overline G_A(\alpha,\beta)+ \Ex_\alpha\left( \one_{\{\tau_A<+\infty\}} G(S_{\tau_A},\beta)\right)+\Ex_\alpha\left(\tau_A \one_{\{\tau_A<+\infty\}}\Gamma(S_{\tau_A},\beta)\right).
  \end{align*}
Now use $H_{A^{\mathrm c}}=\tau_A$ to complete the proof.
\end{proof}
It now becomes important to understand how $G_{V_N}$ and $\overline G_{V_N}=:\overline G_N$ differ. The non-trivial answer to this query says that essentially in the bulk they turn out to be close. This was derived in the following
\begin{theorem}[{\citet[Corollary 2.5.5]{Kurt_thesis}}] \label{thm:noemi} For $d\geq 4$ and $0<\delta<1$, there exists a constant $c_d=c_d(\delta)$ such that
\[
{\sup_{\alpha,\,\beta\in V_N^\delta}\left| G_N(\alpha,\beta)-\overline G_N(\alpha, \beta)\right|}\le c_d{N^{\frac{4-d}{d}}}   .
\]
\end{theorem}
\subsection{Proofs of Assumptions for the membrane model}
We are now ready to prove the results for the MM by showing that the assumptions of Theorem~\ref{thm:zero} hold. For the membrane model we take $s_N= (\log N)^{T}$ with $T>(2+\theta)/(d-4)$ for some $\theta>0$.

\begin{proof}[Proof of Assumptions in Theorem~\ref{thm:zero}]
Condition~(\ref{item:decreasing}) follow immediately from discussions in Subsection~\ref{subsec:MM} and Lemma~\ref{lemma: covariance:mm}.

{\bf Condition (\ref{item:less}): }
Observe that from the proof of \citet[Proposition 2.1.1]{Kurt_thesis} it follows that $G_N(\alpha,\alpha)\le \overline G_{N+1}(\alpha,\alpha)$. By the representation $\overline G_{N+1}(\alpha,\alpha)=\sum_{\beta\in V_{N+1}} \Gamma_{N+1}(\alpha, \beta)\Gamma_{N+1}(\beta,\alpha)$ and the fact that $\Gamma_N(\alpha,\beta)\le \Gamma(\alpha,\beta)$, the upper bound follows.
Now we show that there exists a constant $C= C(d,\delta)$ such that for $\alpha, \beta\in V_N^\delta$,
$$G_N(\alpha,\beta)\ge G(\alpha,\beta)- C N^{\frac{4-d}{d}}.$$
Let us write $G_N(\alpha,\beta)= G(\alpha,\beta)-\bR (\alpha,\beta)$ where $\bR(\alpha,\beta):= G(\alpha,\beta)-G_N(\alpha,y) $. First we look at the error $\bR(\alpha,\beta)$. Note that
\begin{equation}\label{eq:BR}
\bR(\alpha,\beta)= ( G(\alpha,\beta)- \overline G_N(\alpha,\beta))+ (\overline G_N(\alpha,\beta)- G_N (\alpha,\beta)).
\end{equation}
The second summand in~\eqref{eq:BR}, thanks to Theorem~\ref{thm:noemi}, is bounded by $c(d,\delta) N^{\frac{4-d}{d}}$. To tackle the first term we use Proposition~\ref{prop:G_decomposition}:
\begin{align}
G(\alpha,\beta)-\overline G_N(\alpha,\beta) &\nonumber\\
&= \sum_{\gamma\in \partial V_N} \mathrm P_\alpha\left[ H_{\Z^d\setminus V_N}<+\infty, S_{H_{\Z^d\setminus V_N}}=\gamma\right] G(\gamma,\beta)\nonumber\\
&+\sum_{\gamma\in \partial V_N} \Ex_\alpha\left[ H_{\Z^d\setminus V_N}\one_{\{H_{\Z^d\setminus V_N}<+\infty\}} \one_{\left\{S_{H_{\Z^d\setminus V_N}}=\gamma\right\}}\right]\Gamma(\gamma,\beta).\label{eq:t}
\end{align}

Note that since $\alpha,\beta \in V_N^\delta$ we have from Lemma~\ref{lemma: covariance:mm} that, if $N$ is large enough,
$$\sum_{\gamma\in \partial V_N} \mathrm P_\alpha\left[ H_{\Z^d\setminus V_N}<\infty, S_{H_{\Z^d\setminus V_N}}=\gamma\right] G(\gamma,\beta)\le 2\eta_2\sup_{\gamma\in \partial V_N} \|\gamma-\beta\|^{4-d}\le C \left(\delta N^{1/d}\right)^{4-d}.$$

For the other term in \eqref{eq:t} first note that $M_n:=\|S_n\|^2-n$ is a martingale \cite[Exercise 1.4.3]{Lawler} and that $H_{\Z^d\setminus V_N}=\tau_{V_N}=:\tau_N$ under $\mathrm P_\alpha$. Following the idea of \citet{Lawler} before Equation~(1.21), we apply the optional sampling theorem which yields
\[
\Ex_\alpha\left[\|S_{\tau_N}\|^2\right]-\Ex_\alpha\left[ \tau_{N}\right]=\Ex_\alpha\left[M_{\tau_N}\right]=\Ex_\alpha\left[M_0\right]=\|\alpha\|^2
\]
so that
\begin{align*}
\Ex_\alpha&\left[ H_{\Z^d\setminus V_N}\one_{\left\{H_{\Z^d\setminus V_N}<+\infty\right\}}\right]\le \Ex_\alpha\left[ \tau_{N}\right]=\Ex_\alpha\left[\|S_{\tau_{N}}\|^2\right]-\|\alpha\|^2.
\end{align*}
Since $S_{\tau_{N}}\in \partial V_N$
\eq{}\label{eq:marti}
\Ex_\alpha\left[\|S_{\tau_{N}}\|^2\right]-\|\alpha\|^2\le c(d,\,\delta)N^{2/d}.
\eeq{}
Using this we have
\begin{align*}
\sum_{\gamma \in \partial V_N} \Ex_\alpha&\left[ H_{\Z^d\setminus V_N}\one_{\left\{H_{\Z^d\setminus V_N}<+\infty\right\}} \one_{\left\{S_{H_{\Z^d\setminus V_N}}=\gamma\right\}}\right]\Gamma(\gamma,\beta)\\
&\le \sup_{\gamma\in \partial V_N}\Gamma(\gamma,\beta)\Ex_\alpha\left[ H_{\Z^d\setminus V_N}\one_{\left\{H_{\Z^d\setminus V_N}<+\infty\right\}} \right] \\
&\le C(d,\,\delta)  N^{\frac{2-d}{d}}N^{2/d}\le C(d,\,\delta)N^{\frac{4-d}{d}}.
\end{align*}
Combining these inequalities we have that for $\alpha\in V_N^\delta$, $\beta\in V_N^\delta$,
$\bR(\alpha,\beta) \le C(\delta, d) N^{\frac{4-d}{d}}$ and hence Condition~(\ref{item:less}) is proved.

{\bf Condition~(\ref{item:last}) under $\bprob$}: fix $K:=V_N\setminus B(\alpha,\,s_N)$, $U:=\Z^d\setminus K=(\Z^d\setminus V_N)\cup B(\alpha,\,s_N)$. Note that in the decomposition in Lemma~\ref{prop: MP} we have that $\mu_\alpha$ and $\psi_\alpha$ are independent. Also we know that $\psi_\alpha$ is distributed as $\bprob_U$ by Proposition~\ref{prop: MP}. Hence we have
\begin{align}
\var{\mu_\alpha}&=\var{\vr_\alpha}-\var{\psi_\alpha}\\
&= G(0)- G_U(\alpha,\alpha)\nonumber\\
&=\left(G(0)-\overline G_U(\alpha,\alpha)\right)+ \left(\overline G_U(\alpha,\alpha)- G_U(\alpha,\alpha)\right)=: E_1+E_2\label{eq:chiarini}
\end{align}
where $\overline G_U(\alpha,\alpha)=\sum_{k\ge 0}(k+1) \mathrm P_\alpha[ S_k=\alpha, k< H_K]$.  First we deal with $E_1$ using Proposition~\ref{prop:G_decomposition}:
\begin{align*}
&G(0)-\overline G_U(\alpha,\alpha)=\\
&=\sum_{\gamma\in K} \mathrm P_\alpha[ H_K<+\infty, S_{H_K}=\gamma] G(\gamma,\alpha)+\sum_{\gamma\in K} \Ex_\alpha\left[ H_K\one_{\left\{H_K<+\infty\right\}} \one_{\left\{S_{H_K}=\gamma\right\}}\right]\Gamma(\gamma,\alpha).
\end{align*}
One notes that
\begin{equation}\label{eq: E_1}
E_1\le \sup_{\gamma\in K} G(\gamma,\alpha)+\sup_{\gamma\in K} \Gamma(\gamma,\alpha) \Ex_\alpha\left[ H_K\one_{\left\{H_K<+\infty\right\}} \right].\end{equation}
Since the SRW is transient in $d\ge 5$ we have $\tau_{B_\alpha}<+\infty$ $\mathrm{P}_\alpha$-almost surely.  Using the same idea which led to \eqref{eq:marti} we get
\begin{align*}
\Ex_\alpha\left[ H_K\one_{\{H_K<+\infty\}}\right]&=\Ex_\alpha\left[ \tau_{B_\alpha}\one_{\{\tau_{B_\alpha}<+\infty\}}\right]\le C(d,\,\delta)s_N^{2}.
\end{align*}
Hence plugging this estimate in~\eqref{eq: E_1} one gets that
$$E_1\le C\left(s_N^{4-d} + s_N^{2-d} s_N^{2}\right)= C s_N^{4-d}.$$
Note that this gives a bound on $E_1$ and hence we are left with $E_2= \left(\overline G_U(\alpha,\alpha)- G_U(\alpha,\alpha)\right)$. Note that $B_\alpha\subsetneq U=(\Z^d\setminus V_N)\cup B(\alpha,s_N)$. Since we know that variances increase as shown in Lemma~\ref{lemma:monotonicity_variances} the inequality $G_{B_\alpha}(\alpha,\alpha)\le G_U(\alpha,\alpha)$ holds. Actually it follows that for $\alpha \in V_N^\delta$, we have $\overline G_U(\alpha,\alpha)=\overline G_{B_\alpha}(\alpha,\alpha)$. Indeed, using the fact that $B(\alpha,s_N)\subsetneq V_N$ (that is, for a walk starting at $\alpha$, $\tau_U= \tau_{B_\alpha}$) we have
 $$\overline G_U(\alpha, \alpha)= \overline G_{B_\alpha}(\alpha,\alpha)=\sum_{n\ge 0} (n+1) \mathrm P_\alpha\left[ S_n=\alpha, n< \tau_{B_\alpha}\right]$$
We need to now look at $\overline G_{B_\alpha}(\alpha,\alpha)- G_{B_\alpha}(\alpha,\alpha)$.  Let us denote the bulk of $B_\alpha$ as
$$B(\alpha,s_N)^\delta:=\{ \beta\in B(\alpha,s_N):\, \|\beta-\gamma\| \ge \delta s_N,\,\text{for all}\;\gamma\in B(\alpha, s_N)^{\mathrm c}\}.$$
Observe that $\alpha\in B(\alpha,s_N)^\delta$.  Hence using Theorem~\ref{thm:noemi},
\begin{align*}
&E_2= \overline G_{U}(\alpha,\alpha)- G_U(\alpha,\alpha)\\
&\le \overline G_{ B_\alpha}(\alpha,\alpha)- G_{B_\alpha}(\alpha,\alpha)\le \sup_{\gamma \in B(\alpha,s_N)^\delta}\left|\overline G_{ B(\alpha,s_N)}(\alpha,\gamma)- G_{B(\alpha,s_N)}(\alpha,\gamma)\right|\\
&=\sup_{\gamma \in B(0,\,s_N)^\delta}\left|\overline G_{ B(0,\,s_N)}(0,\,\gamma)- G_{B(0,\,s_N)}(0,\,\gamma)\right|\le C(d,\,\delta) s_N^{4-d}.
\end{align*}
Recalling \eqref{eq:chiarini} we have just proved that $\var{\mu_\alpha}\le C(d,\,\delta)s_N^{4-d}$. Since $s_N=(\log N)^T$ with $T>(2+\theta)/(d-4)$,  Condition~(\ref{item:last})  follows.

{\bf Condition~(\ref{item:last}) under $\bprob_N$}: the proof is very similar to the argument under $\bprob$ and hence we briefly sketch it. Note that $\var{\mu_\alpha}= G_N(\alpha,\alpha)- G_{B_\alpha}(\alpha,\alpha)$. As before from Theorem~\ref{thm:noemi} it follows that $\left|G_N(\alpha,\alpha)- \overline G_N(\alpha,\alpha)\right|\le c_d N^{4-d}$ and $\left|G_{B_\alpha}(\alpha,\alpha)-\overline G_{B_\alpha}(\alpha,\alpha)\right|\le c_d s_N^{4-d}$, therefore it is enough to bound $\left|\overline G_N(\alpha,\alpha)- \overline G_{B_\alpha}(\alpha,\alpha)\right|$. Note that $\tau_{B_\alpha}\le \tau_{V_N}$, $\mathrm P_\alpha$-almost surely. Hence we have that
\begin{align*}
\left|\overline G_N(\alpha,\alpha)- \overline G_{B_\alpha}(\alpha,\alpha)\right|&=\left|\sum_{m\ge 0}(m+1)\mathrm P_\alpha[ S_m=\alpha, m< \tau_{V_N}]-\sum_{m\ge 0}(m+1) \mathrm P[ S_m=\alpha, \,m<\tau_{B_\alpha}]\right|\\
&=\sum_{m\ge 0} (m+1) \mathrm P_\alpha[ S_m=\alpha,\, \tau_{B_\alpha}\le m< \tau_{V_N}]\\
&=\Ex_\alpha\left[ \left(\sum_{m=0}^{+\infty} (m+ \tau_{B_\alpha}+1)\one_{\left\{S_m= \alpha, \,m < \tau_{V_N}\right\}}\right)\circ \Theta_{\tau_{B_\alpha}}\right]\\
&\stackrel{MP}{=}\sum_{\gamma\in \partial B_\alpha} \mathrm P_\alpha[ S_{ \tau_{B_\alpha}}=\gamma]\overline G_N(\gamma,\alpha)+ \sum_{\gamma\in \partial B_\alpha}\Ex_\alpha\left[ \tau_{B_\alpha} \one_{\left\{S_{\tau_{B_\alpha}}=\gamma\right\}}\right] \Gamma_N(\gamma,\alpha)\\
&\le \sup_{\gamma\in \partial B_\alpha} \overline G_N(\gamma,\alpha)+\sup_{\gamma\in \partial B_\alpha}\Gamma(\gamma, \alpha) \Ex_\alpha\left[ \tau_{B_\alpha}\right]\\
&\le C_d s_N^{4-d}.
\end{align*}
In the last line we have used previous estimates to conclude that $\Ex_\alpha[ \tau_{B_\alpha}]\le c s_N^{2}$.
\end{proof}
\begin{proof}[Proof of Proposition~\ref{prop: MP} (b)]
Recall $K=V_N\setminus B_\alpha$, $U=\Z^d\setminus K$ in what follows. First observe that as in proof of Theorem~\ref{thm:zero} it follows that $\mu_\alpha$ is $\mathcal F_K$-measurable and hence is independent of $\psi_\alpha=\vr_\alpha-\mathbf E[ \vr_\alpha|\mathcal F_K]$.

To show the next part of the lemma, we follow the proof of Lemma~1.2 of \cite{PFASS} and hence we need to show that
\begin{equation}\label{eq:MP:main identity}
\mathbf E\left[\one_A( (\psi_\alpha)_{\alpha\in U})\right]=\mathbf E_{U}\left[\one_A((\vr_\alpha)_{\alpha\in U})\right] \qquad \forall \, A\in \mathcal A_U
\end{equation}
where $\mathcal A_U$ is the canonical $\sigma$-algebra on $\R^U$. Let $(\alpha_i)_{0\le i\le n}\in U$ and $V$ be a finite set such that $K\cup \{\alpha_0,\,\ldots,\, \alpha_k\}\subset V\subset \Z^d$. Now by Dynkin's $\pi-\lambda$ theorem it suffices to assume that for $A$ of the form
$$A= A_{\alpha_0}\times \cdots \times A_{\alpha_n}\times \R^{U\setminus \{\alpha_0,\,\ldots,\, \alpha_n\}}$$
where $k\ge 0$ and $A_{\alpha_i}\in \mathcal B(\R)$, $i=0, 1,\,\ldots,\, k$. Consider on $\R^{V}$ the law of the finite membrane model, that is, the centered Gaussian field with covariance $G_{V}(\alpha,\beta)$. We indicate the law by $\bprob_V$ and its canonical process by $\left\{\vr_x^V\right\}_{x\in V}$. Now we define $\mu_\alpha^V=\mathbf E\left[\vr_\alpha^V|\mathcal F^V_K\right]$ where $\mathcal F^V_K=\sigma(\vr_\beta^V: \beta\in K)$ and also denote $\psi_\alpha^V=\vr_\alpha^V-\mu_\alpha^V$. Finally, we claim \eqref{eq:MP:main identity} holds because of the following statement:
\begin{claim}[DLR equation for an infinite set]\label{claim:convergence}
As $V\uparrow \Z^d$ we have
\eq{}\label{eq:sec}\mathbf E_{V\setminus K}[ \one_{A^V}((\vr_\alpha)_{\alpha\in V\setminus K})]\to \mathbf E_U\left[ \one_A((\vr_\alpha)_{\alpha\in U})\right] \eeq{}
and
\eq{}\label{eq:first}\mathbf E_V\left[ \one_{A^V}\left( \left(\psi_\alpha^V\right)_{\alpha\in V\setminus K}\right)\right]\to \mathbf E\left[ \one_A( (\psi_\alpha)_{\alpha\in U})\right]\eeq{}
where $A^V= A_{\alpha_0}\times \cdots \times A_{\alpha_n}\times \R^{V\setminus K\cup\{\alpha_0,\,\ldots,\, \alpha_n\}}.$ In particular, the limiting law $\bprob_U$ exists.
\end{claim}
Before we embark on showing the Claim, let us clarify its meaning and why it implies~\eqref{eq:MP:main identity} (which is nothing but the DLR equation for an infinite set $U$, see \citet[Section 6.2.1]{veleniknotes}). Since $U$ is infinite, this is, a priori, a non-trivial statement. Recall that, due to Proposition~\ref{prop: MP} (a), for a finite volume $V$
\eq{}\label{eq:DLRfin}
\mathbf E_{V\setminus K}[ \one_{A^V}((\vr_\alpha)_{\alpha\in V\setminus K})]=\mathbf E_V\left[ \one_{A^V}\left( \left(\psi_\alpha^V\right)_{\alpha\in V\setminus K}\right)\right].
\eeq{}
We are now letting $V\uparrow \Z^d$ on both sides of \eqref{eq:DLRfin}, so that:
\begin{center}
\begin{tikzpicture}
  \matrix (m) [matrix of math nodes,row sep=3em,column sep=4em,minimum width=2em]
  {
  \mathbf E_{V\setminus K}[ \one_{A^V}((\vr_\alpha)_{\alpha\in V\setminus K})]& \mathbf E_V\left[ \one_{A^V}\left( \left(\psi_\alpha^V\right)_{\alpha\in V\setminus K}\right)\right] \\
    \mathbf E_U\left[ \one_A((\vr_\alpha)_{\alpha\in U})\right]& \mathbf E\left[ \one_A( (\psi_\alpha)_{\alpha\in U})\right] \\};
  \path[-stealth]
    (m-1-1) edge node [left] {\eqref{eq:sec}} (m-2-1)
            edge [draw=none] node [below] {\eqref{eq:DLRfin}} node[sloped,auto=false] {$=$} (m-1-2)
    (m-2-1.east|-m-2-2) edge[draw=none] node [below] {\eqref{eq:MP:main identity}}
            node[sloped,auto=false] {$=$} node [above] {!}(m-2-2)
    (m-1-2) edge node [right] {\eqref{eq:first}} (m-2-2)
            ;
\end{tikzpicture}
\end{center}
(where ``!'' indicates the equality we must still show).
This completes the proof of \eqref{eq:MP:main identity} and hence the proof of Proposition~\ref{prop: MP} (b).
\end{proof}

\begin{proof}[Proof of Claim~\ref{claim:convergence}] We discuss only \eqref{eq:sec}, since \eqref{eq:first} follows similarly.
We use \citet[Proposition 2.1]{Giacomin} and the uniform bound on $G_{V\setminus K}(\alpha,\,\alpha)$, $\alpha\in \Z^d$ to obtain that the measures $\{\bprob_{V\setminus K} \}_{V\Subset \Z^d}$ are tight. Therefore it suffices to show that $G_{V\setminus K}$ converges pointwise to a bounded limit. An application of \citet[Theorems 13.24, 13.26]{Georgii} yields that
\[G_U(\alpha,\,\beta):=\lim_{V\uparrow\Z^d}G_{V\setminus K}(\alpha,\,\beta)
\]
exists. One sees that $G_U$ satisfies the following boundary value problem: for $\alpha\in U$
\[
\left\{\begin{array}{lr}
\Delta^2 G_U(\alpha,\beta)=\delta(\alpha,\, \beta),& \beta \in U\nonumber \\
G_U(\alpha,\beta)=0, & \beta \in \partial_2 U.\nonumber
\end{array}\right.
\]
This shows that the centered Gaussian random measure with covariance $G_U$ exists and is the weak limit of $\bprob_{V\setminus K}$ as $V\uparrow \Z^d$. In particular \eqref{eq:sec} holds true.
\end{proof}


\subsection{Proof of assumptions for massive DGFF}
We introduced the massive discrete Gaussian free field in Subsection~\ref{subsec:intro:massive}. Here we briefly point out the how the assumptions of Theorem~\ref{thm:zero} are satisfied for this model. We denote by $\bprob_{\vartheta, N}$ the law of the massive free field with zero boundary conditions outside $V_N$, namely $\vr=(\vr_\alpha)_{\alpha\in\Z^d}$ under $\bprob_{\vartheta, N}$ is the centered Gaussian field with covariance structure
\[
\mathbf E_{\vartheta, N}\left[ \vr_\alpha\vr_\beta\right]= g_{\vartheta, V_N^{\mathrm c}}(\alpha,\beta)
\]
where recall $g_{\vartheta, V_N^{\mathrm c}}(\alpha,\beta)=\sum_{m\ge 0} (1-\vartheta)^m \prob^\alpha_0\left[ S_m=\beta, \,m< H_{V_N^{\mathrm c}}\right]$. Hence using the representation it follows that $g_{\vartheta, V_N^{\mathrm c}}(\alpha,\beta)\to g_\vartheta(\alpha,\beta)$ as $V_N\uparrow \Z^d$.  Consequently $\bprob_{\vartheta,N}\to \bprob_\vartheta$. For detailed proofs of the above facts see \citet[Section 8.5]{veleniknotes}.

To verify the assumptions the decay of the covariance is needed.
\begin{lemma}[{\citet[Theorem 7.48]{veleniknotes}}]\label{lemma:cov:massive}
There exists a constant $c(\vartheta)>0$ such that for $\alpha\neq \beta\in \Z^d$,
$$\exp\left(-c(\vartheta)\|\alpha-\beta\|\right)\le g_\vartheta(\alpha,\beta)\le \frac1{\vartheta}\exp\left(-c(\vartheta)\|\alpha-\beta\|\right).$$
\end{lemma}
%
Note that this implies that the covariance decays exponentially to $0$ and hence Assumption~\eqref{item:decreasing} follows. To show the next assumption we take $\alpha,\beta\in V_N^\delta$. We observe that by~\eqref{eq:massiveMP:green} and Lemma~\ref{lemma:cov:massive} we have
\begin{align}
g_{\vartheta, V_N^\mathrm{c}}(\alpha,\beta)&= g_\vartheta(\alpha,\beta)-\sum_{\gamma\in \partial V_N}\mathrm{P}^\alpha_\vartheta\left[ H_{V_N^\mathrm{c}}<+\infty, S_{H_{V_N^\mathrm{c}}}=\gamma\right] g_\vartheta(\gamma,\beta)\nonumber\\
&\ge g_\vartheta(\alpha,\beta)-\sup_{\gamma\in \partial V_N} g_\vartheta(\gamma,\beta)
\ge g_\vartheta(\alpha,\beta)-\frac1{\vartheta} \exp\left(-c(\vartheta,\,\delta) N^{1/d}\right).\label{eq:omer}
\end{align}
Hence Assumption~\eqref{item:less} immediately follows.
Let $K=V_N\setminus B_\alpha$ as in Assumption~\eqref{item:last}. First note that by the Green's function decomposition we have for $\beta\in K$
\begin{align*}
g_\vartheta(\alpha,\beta)&=g_{\vartheta, K}(\alpha, \beta)+\sum_{\gamma\in K} \mathrm{P}^\alpha_\vartheta[ H_K<+\infty, S_{H_K}=\gamma] g_\vartheta(\gamma, \beta)\\
&=\sum_{\gamma\in K} \mathrm{P}^\alpha_\vartheta[ H_K<+\infty, S_{H_K}=\gamma] g_\vartheta(\gamma, \beta)
\end{align*}
where we have used the fact that $\alpha\notin K$ and $\beta\in K$ implies $g_{\vartheta, K}(\alpha, \beta)=0$. Using the above identity we have
\begin{align}
\var{\mu_\alpha}&=\sum_{\beta,\gamma\in K} \mathrm{P}^\alpha_\vartheta[ H_K<+\infty, S_{H_K}=\beta]\mathrm{P}^\alpha_\vartheta[ H_K<\infty, S_{H_K}=\gamma]g_\vartheta(\beta, \gamma)\nonumber\\
&=\sum_{\beta\in K} \mathrm{P}^\alpha_\vartheta[ H_K<+\infty, S_{H_K}=\beta] g_\vartheta(\alpha, \beta)\nonumber\\
&\le \sup_{\beta\in K} g_\vartheta(\alpha, \beta)\le \sup_{\beta\in K} \frac1{\vartheta} \e^{-c(\vartheta) \|\alpha-\beta\|}\le C\e^{-c(\vartheta) s_N}.\label{eq:nick}
\end{align}
Now note that taking $s_N=\log N$ we have $(\log N)^{2+\theta}\var{\mu_\alpha}\le (\log N)^{2+\theta}N^{-c(\vartheta)}\to 0$ for any $\theta>0$ and uniformly for all $\alpha\in V_N^\delta$. Hence Assumption~\eqref{item:last} follows.

\subsection{Proof of assumptions for fractional fields}
In this subsection we point out how the fractional field described in Subsection \ref{subsec:DFGF} satisfies the assumptions in Theorem~\ref{thm:zero}. Looking at the representation in \eqref{eq:dfgf:greenrep} it follows that $G_{s,V_N}(\alpha,\beta)\to G_s(\alpha, \beta)$ and hence $\bprob_{s, N}\to \bprob_s$. Also note that Assumption~\eqref{item:decreasing} follows from~\eqref{eq:dfgf:cov}. Let us prove Assumption~(\ref{item:less}). Since we have a lot of freedom in tuning parameters for the model, assume $s_N=(\log N)^{\frac{\xi}{d-s}}$, $\xi>2$, to fix ideas. We can then take $N$ large enough so that for all $\beta\in \Z^d$ and $\gamma$ such that $\|\beta-\gamma\|\ge s_N$, \eqref{eq:dfgf:cov} yields
\eq{}\label{eq:Vladas}
\frac12 w_{s,\,d}\|\beta-\gamma\|^{s-d}\le G_s(\beta,\,\gamma)\le \frac32w_{s,\,d}\|\beta-\gamma\|^{s-d}.
\eeq{}
With this choice, since $s_N$ is much smaller than $\delta N$, for all $\beta\in V_N^\delta$ and $\gamma\in \partial V_N$ the inequalities hold. Observe that by the decomposition of the Green's function \eqref{eq:green:MP} we have, for $\alpha,\beta\in V_N^\delta$,
\begin{align*}
G_{s, V_N}(\alpha,\beta)&=G_s(\alpha,\beta)-\sum_{\gamma\in \Z^d\setminus V_N} \prob_\alpha\left[ \tau_{V_N}<+\infty, S_{\tau_{V_N}}=\gamma\right]G_s(\gamma,\beta).
\end{align*}
Therefore analogous computations to \eqref{eq:omer} (with the appropriate changes due to the representation \eqref{eq:green:MP}) yield
$$
G_{s, V_N}(\alpha,\,\beta)\ge G_s(\alpha,\,\beta)-\sup_{\gamma\in \Z^d\setminus V_N}G_s(\gamma,\,\beta)\ge G_s(\alpha,\,\beta)-c(\delta) N^{\frac{s-d}{d}}.
$$
Moreover as in \eqref{eq:nick} one can see that
$$
\var{\mu_\alpha}\le \sup_{\beta \in K}G_s(\alpha,\,\beta)\le  \frac{3}{2}\sup_{\beta \in K}  w_{s,\,d}\|\beta-\alpha\|^{s-d}\le
\frac{3}{2} w_{s,\,d} s_N^{s-d}=\frac{3}{2} w_{s,\,d}(\log N)^{-\xi}.
$$
Any $\theta<\xi-2$ allows to say that $\lim_{N\to+\infty}(\log N)^{2+\theta}\var{\mu_\alpha}= 0$ uniformly for $\alpha\in V_N^\delta$. Hence the above arguments show that the fractional fields satisfy Assumptions~\eqref{item:less} and \eqref{item:last}.

\section*{Acknowledgements}
The first author's research was supported by RTG 1845. The last author's research was supported by Cumulative Professional Development Allowance from Ministry of Human Resource Development, Government of India, and Department of Science and Technology, Inspire funds. He also acknowledges the hospitality of WIAS Berlin where part of the present work was carried out. The authors would like to thank Noemi Kurt for clarifying us some details of the membrane model and Ofer Zeitouni for asking a question that led to the draft of the present paper. We thank two anonymous referees for their thorough review and highly appreciate the comments and 
suggestions which contributed to improving the quality of the publication.
\bibliographystyle{abbrvnat}
\bibliography{literatur}

\begin{thebibliography}{34}
\providecommand{\natexlab}[1]{#1}
\providecommand{\url}[1]{\texttt{#1}}
\expandafter\ifx\csname urlstyle\endcsname\relax
  \providecommand{\doi}[1]{doi: #1}\else
  \providecommand{\doi}{doi: \begingroup \urlstyle{rm}\Url}\fi

\bibitem[Arratia et~al.(1989)Arratia, Goldstein, and Gordon]{AGG}
R.~Arratia, L.~Goldstein, and L.~Gordon.
\newblock {Two Moments Suffice for Poisson Approximations: The Chen-Stein
  Method}.
\newblock \emph{Ann. Probab.}, 17\penalty0 (1):\penalty0 9--25, 01 1989.
\newblock \doi{10.1214/aop/1176991491}.
\newblock URL \url{http://dx.doi.org/10.1214/aop/1176991491}.

\bibitem[Berman(1964)]{Berman}
S.~M. Berman.
\newblock Limit theorems for the maximum term in stationary sequences.
\newblock \emph{Ann. Math. Statist.}, 35\penalty0 (2):\penalty0 502--516, 06
  1964.
\newblock \doi{10.1214/aoms/1177703551}.
\newblock URL \url{http://dx.doi.org/10.1214/aoms/1177703551}.

\bibitem[{Biskup} and {Louidor}(2013)]{BisLou}
M.~{Biskup} and O.~{Louidor}.
\newblock {Extreme local extrema of two-dimensional discrete Gaussian free
  field}.
\newblock \emph{ArXiv e-prints}, June 2013.
\newblock \url{http://adsabs.harvard.edu/abs/2013arXiv1306.2602B}.

\bibitem[Bolthausen et~al.(1995)Bolthausen, Deuschel, and Zeitouni]{BDZ95}
E.~Bolthausen, J.-D. Deuschel, and O.~Zeitouni.
\newblock Entropic repulsion of the lattice free field.
\newblock \emph{Communications in mathematical physics}, 170\penalty0
  (2):\penalty0 417--443, 1995.

\bibitem[Bolthausen et~al.(2001)Bolthausen, Deuschel, and Giacomin]{BDG}
E.~Bolthausen, J.~D. Deuschel, and G.~Giacomin.
\newblock Entropic repulsion and the maximum of the two-dimensional harmonic
  crystal.
\newblock \emph{The Annals of Probability}, 29\penalty0 (4):\penalty0
  1670--1692, 2001.

\bibitem[{Bramson} et~al.(2013){Bramson}, {Ding}, and {Zeitouni}]{BrDiZe}
M.~{Bramson}, J.~{Ding}, and O.~{Zeitouni}.
\newblock {Convergence in law of the maximum of the two-dimensional discrete
  Gaussian free field}.
\newblock \emph{ArXiv e-prints}, Jan. 2013.
\newblock \url{http://arxiv.org/abs/1301.6669v4}.

\bibitem[Caputo(2000)]{caputo_thesis}
P.~Caputo.
\newblock \emph{Harmonic crystals: statistical mechanics and large deviations}.
\newblock PhD thesis, TU-Berlin 2000, http://edocs. tu-berlin. de/diss/index.
  html, 2000.
\newblock URL
  \url{http://webdoc.sub.gwdg.de/ebook/diss/2003/tu-berlin/diss/2000/caputo_pietro.pdf}.

\bibitem[Chiarini et~al.(2015)Chiarini, Cipriani, and Hazra]{CCH2015b}
A.~Chiarini, A.~Cipriani, and R.~S. Hazra.
\newblock {A note on the extremal process of the supercritical Gaussian Free
  Field }.
\newblock \emph{Electron. Commun. Probab.}, 20:\penalty0 10 pp., 2015.
\newblock \doi{10.1214/ECP.v20-4332}.
\newblock URL \url{http://dx.doi.org/10.1214/ECP.v20-4332}.

\bibitem[Chiarini et~al.(2016)Chiarini, Cipriani, and Hazra]{CCH2015}
A.~Chiarini, A.~Cipriani, and R.~S. Hazra.
\newblock {Extremes of the supercritical Gaussian Free Field}.
\newblock \emph{Alea: Latin American Journal of Probability and Mathematical
  Statistics}, 13:\penalty0 711--724, 2016.

\bibitem[Cipriani(2013)]{Cip13}
A.~Cipriani.
\newblock High points for the membrane model in the critical dimension.
\newblock \emph{Electron. J. Probab.}, 18:\penalty0 no. 86, 1--17, 2013.
\newblock ISSN 1083-6489.
\newblock \doi{10.1214/EJP.v18-2750}.
\newblock URL \url{http://ejp.ejpecp.org/article/view/2750}.

\bibitem[Daviaud(2006)]{Daviaud}
O.~Daviaud.
\newblock {Extremes of the discrete two-dimensional Gaussian Free Field}.
\newblock \emph{The Annals of Probability}, 34\penalty0 (3):\penalty0 962--986,
  2006.

\bibitem[Ding et~al.(2015)Ding, Roy, and Zeitouni]{ding:roy:ofer}
J.~Ding, R.~Roy, and O.~Zeitouni.
\newblock Convergence of the centered maximum of log-correlated gaussian
  fields.
\newblock \emph{arXiv preprint arXiv:1503.04588}, 2015.

\bibitem[Drewitz and Rodriguez(2015)]{DrePF}
A.~Drewitz and P.-F. Rodriguez.
\newblock High-dimensional asymptotics for percolation of gaussian free field
  level sets.
\newblock \emph{Electron. J. Probab.}, 20:\penalty0 no. 47, 1--39, 2015.
\newblock ISSN 1083-6489.
\newblock \doi{10.1214/EJP.v20-3416}.
\newblock URL \url{http://ejp.ejpecp.org/article/view/3416}.

\bibitem[Dynkin(1980)]{Dyn80}
E.~B. Dynkin.
\newblock Markov processes and random fields.
\newblock \emph{Bull. Amer. Math. Soc. (N.S.)}, 3\penalty0 (3):\penalty0
  975--999, 11 1980.
\newblock URL \url{http://projecteuclid.org/euclid.bams/1183547683}.

\bibitem[Embrechts et~al.(2013)Embrechts, Kl{\"u}ppelberg, and
  Mikosch]{EmbKluMik}
P.~Embrechts, C.~Kl{\"u}ppelberg, and T.~Mikosch.
\newblock \emph{{Modelling Extremal Events: for Insurance and Finance}}.
\newblock Stochastic Modelling and Applied Probability. Springer Berlin
  Heidelberg, 2013.
\newblock ISBN 9783540609315.
\newblock URL \url{https://books.google.it/books?id=BXOI2pICfJUC}.

\bibitem[Friedli and Velenik(2015)]{veleniknotes}
S.~Friedli and Y.~Velenik.
\newblock {Equilibrium Statistical Mechanics of Classical Lattice Systems: a
  Concrete Introduction}.
\newblock 2015.
\newblock URL \url{http://www.unige.ch/math/folks/velenik/smbook/index.html}.

\bibitem[Funaki(2005)]{Funaki}
T.~Funaki.
\newblock Stochastic interface models. {L}ectures on probability theory and
  statistics.
\newblock \emph{Lect. Notes in Math.}, 1869:\penalty0 103--274, 2005.

\bibitem[Georgii(1988)]{Georgii}
H.~O. Georgii.
\newblock \emph{Gibbs measures and Phase transitions}.
\newblock de Gruyter, Berlin, 1988.

\bibitem[Giacomin(2010)]{Giacomin}
G.~Giacomin.
\newblock Aspects of statistical mechanics of random surfaces, Jan. 2010.
\newblock URL \url{felix.proba.jussieu.fr/pageperso/giacomin/pub/IHP.ps}.

\bibitem[Janson(1997)]{Jan97}
S.~Janson.
\newblock \emph{Gaussian Hilbert Spaces}.
\newblock Cambridge Tracts in Mathematics. Cambridge University Press, 1997.
\newblock ISBN 9780521561280.
\newblock URL \url{http://books.google.de/books?id=heCzfIclio0C}.

\bibitem[Kallenberg(1983)]{KallenbergRan}
O.~Kallenberg.
\newblock \emph{Random measures}.
\newblock Akademie-Verlag, 1983.
\newblock ISBN 9780123949608.
\newblock URL \url{https://books.google.de/books?id=bBnvAAAAMAAJ}.

\bibitem[Kurt(2007)]{Kurt_d5}
N.~Kurt.
\newblock {Entropic repulsion for a class of Gaussian interface models in high
  dimensions.}
\newblock \emph{Stochastic Processes Appl.}, 117\penalty0 (1):\penalty0 23--34,
  2007.
\newblock \doi{10.1016/j.spa.2006.05.011}.

\bibitem[Kurt(2008)]{Kurt_thesis}
N.~Kurt.
\newblock \emph{{Entropic repulsion for a Gaussian membrane model in the
  critical and supercritical dimension}}.
\newblock PhD thesis, University of Zurich, 2008.
\newblock URL \url{https://www.zora.uzh.ch/6319/3/DissKurt.pdf}.

\bibitem[Kurt(2009)]{Kurt_d4}
N.~Kurt.
\newblock {Maximum and entropic repulsion for a Gaussian membrane model in the
  critical dimension}.
\newblock \emph{The Annals of Probability}, 37\penalty0 (2):\penalty0 687--725,
  2009.

\bibitem[Lawler(1991)]{Lawler}
G.~Lawler.
\newblock \emph{Intersections of random walks}.
\newblock Birkh{\"a}user, Boston, 1991.

\bibitem[P\'olya(1920)]{Polya}
G.~P\'olya.
\newblock {\"Uber den zentralen Grenzwertsatz der Wahrscheinlichkeitsrechnung
  und das Momentenproblem}.
\newblock \emph{Mathematische Zeitschrift}, 8\penalty0 (3-4):\penalty0
  171--181, 1920.
\newblock ISSN 0025-5874.
\newblock \doi{10.1007/BF01206525}.
\newblock URL \url{http://dx.doi.org/10.1007/BF01206525}.

\bibitem[Resnick(1987)]{Resnick}
S.~Resnick.
\newblock \emph{Extreme Values, Regular Variation, and Point Processes}.
\newblock Applied probability : a series of the applied probability trust.
  Springer, 1987.
\newblock ISBN 9783540964810.
\newblock URL \url{http://books.google.it/books?id=DXi1QgAACAAJ}.

\bibitem[Resnick(2007)]{ResnickHeavy}
S.~Resnick.
\newblock \emph{Heavy-Tail Phenomena: Probabilistic and Statistical Modeling}.
\newblock Number Bd. 10 in Heavy-tail Phenomena: Probabilistic and Statistical
  Modeling. Springer, 2007.
\newblock ISBN 9780387242729.
\newblock URL \url{https://books.google.de/books?id=p8uq2QFw9PUC}.

\bibitem[Rodriguez(2015)]{PFmassive}
P.-F. Rodriguez.
\newblock {A $0-1$ law for the massive free field}.
\newblock \emph{arXiv preprint}, 2015.
\newblock URL \url{http://arxiv.org/pdf/1505.08169v1.pdf}.

\bibitem[Rodriguez and Sznitman(2013)]{PFASS}
P.-F. Rodriguez and A.-S. Sznitman.
\newblock Phase transition and level-set percolation for the gaussian free
  field.
\newblock \emph{Communications in Mathematical Physics}, 320\penalty0
  (2):\penalty0 571--601, 2013.
\newblock ISSN 0010-3616.
\newblock \doi{10.1007/s00220-012-1649-y}.
\newblock URL \url{http://dx.doi.org/10.1007/s00220-012-1649-y}.

\bibitem[Sakagawa(2003)]{Sakagawa}
H.~Sakagawa.
\newblock Entropic repulsion for a gaussian lattice field with certain finite
  range interactions.
\newblock \emph{J. Math. Phys.}, 44\penalty0 (7):\penalty0 2939--2951, 2003.

\bibitem[Savage(1962)]{Savage}
I.~Savage.
\newblock { Mill's ratio for multivariate normal distributions}.
\newblock \emph{J. Res. Natl. Bur. Stand. Sec. B: Math.\& Math. Phys.},
  66B:\penalty0 93, 1962.

\bibitem[Sznitman(2012)]{ASS}
A.-S. Sznitman.
\newblock \emph{{Topics in Occupation Times and Gaussian Free Fields}}.
\newblock Zurich Lectures in Advanced Mathematics. American Mathematical
  Society, 2012.
\newblock ISBN 9783037191095.
\newblock URL \url{http://books.google.ch/books?id=RnENO-nQ7TIC}.

\bibitem[Zeitouni(2015)]{ofernotes}
O.~Zeitouni.
\newblock {Gaussian Fields}.
\newblock 2015.
\newblock URL \url{http://www.wisdom.weizmann.ac.il/~zeitouni/notesGauss.pdf}.

\end{thebibliography}

\end{document}